\documentclass[10pt,reqno]{amsart}

\usepackage[T1]{fontenc}
\usepackage{amssymb}
\usepackage{mathtools}      %
\usepackage[hidelinks]{hyperref}   %

\newcommand{\R}{\mathbb{R}}
\newcommand{\Z}{\mathbb{Z}}
\newcommand{\C}{\mathbb{C}}
\newcommand{\E}{\mathbb{E}}
\newcommand{\Hyp}{\mathbb{H}}
\newcommand{\Nil}{\mathrm{Nil}}
\newcommand{\Sol}{\mathrm{Sol}}
\newcommand{\Lie}{\mathrm{Lie}}
\newcommand{\nil}{\mathfrak{nil}}
\newcommand{\SLtilde}{\widetilde{\mathrm{SL}}_2\R}
\newcommand{\Isom}{\mathrm{Isom}}

\newcommand{\diff}{\mathrm{d}}
\newcommand{\trace}{\mathrm{Tr}}
\newcommand{\gm}{\mathrm{geom}}
\newcommand{\PSL}{\mathrm{PSL}}
\newcommand{\SL}{\mathrm{SL}}
\newcommand{\SO}{\mathrm{SO}}
\newcommand{\SE}{\mathrm{SE}}
\newcommand{\se}{\mathfrak{se}}
\newcommand{\id}{\mathrm{id}}
\newcommand{\g}{\mathfrak{g}}
\newcommand{\kk}{\mathfrak{k}}

\newcommand{\so}{\mathfrak{so}}
\newcommand{\sol}{\mathfrak{sol}}
\newcommand{\CS}{\mathrm{CS}}
\newcommand{\CSname}{CS}
\newcommand{\CSvol}[1]{\CS^{\mathrm{vol}}_{#1}}
\newcommand{\CSrot}[1]{\CS^{\mathrm{rot}}_{#1}}
\newcommand{\MC}{\mathrm{MC}}
\newcommand{\LC}{\mathrm{LC}}
\newcommand{\WZ}{\mathrm{WZ}}
\newcommand{\Vol}{\mathrm{Vol}}
\newcommand{\vol}{\mathrm{vol}}
\newcommand{\covol}{\mathrm{covol}}
\newcommand{\cupdot}{\mathbin{\smile}}
\newcommand{\p}{p}
\newcommand{\FrameBundle}{F(M,g)}
\newcommand{\Ad}{\mathrm{Ad}}
\newcommand{\ad}{\mathrm{ad}}
\newcommand{\triv}{\mathrm{triv}}
\renewcommand{\flat}{\mathrm{flat}}

\theoremstyle{plain}
\newtheorem{thm}{Theorem}[section]
\newtheorem*{thm*}{Theorem}
\newtheorem{prop}[thm]{Proposition}
\newtheorem{lem}[thm]{Lemma}
\newtheorem{cor}[thm]{Corollary}

\theoremstyle{definition}
\newtheorem{dfn}[thm]{Definition}
\newtheorem{exm}[thm]{Example}

\theoremstyle{remark}
\newtheorem{rmk}[thm]{Remark}

\theoremstyle{plain}
\newtheorem*{maintheorem}{Main Theorem}

\numberwithin{equation}{section}

\title[Chern--Simons invariants and volumes of representations]{Chern--Simons invariants and volumes of representations in Nil, Sol, and Euclidean geometries}
\author{Taiju Suzuki}
\address{Department of Mathematics, Institute of Science Tokyo, O-okayama, Meguro, Tokyo, 152-8551, Japan}
\email{suzuki.t.6577@m.isct.ac.jp}
\date{}

\begin{document}

\begin{abstract}
    In this paper, we realize volumes of representations as real-valued Chern--Simons invariants in $\Nil$, $\Sol$, and Euclidean geometries.
    To this end, we formulate a Chern--Simons invariant of a pair of connections on a principal bundle that need not be trivial.
    For a connected closed oriented $3$-manifold $M$ and a representation $\rho \colon \pi_1(M) \to G$ into the identity component $G$ of the isometry group of one of these geometries, we construct an auxiliary connection on the associated flat $G$-bundle.
    We show that, for a suitably normalized invariant polynomial, the Chern--Simons invariant of the auxiliary and flat connections equals the volume of the representation.
    For the holonomy representation of a geometric structure, this invariant recovers the Riemannian volume.
    We also compute the Chern--Simons invariant of the Levi-Civita connection for representative closed manifolds in each of these geometries.
\end{abstract}

\maketitle

\tableofcontents

\section{Introduction}
\label{sec:introduction}
Chern--Simons invariants, introduced by Chern and Simons \cite{CS1}, appear in many problems across geometry, topology, and mathematical physics.
Meanwhile, in $3$-manifold topology, the geometrization theorem states that every closed oriented $3$-manifold decomposes into pieces, each of which supports one of the eight Thurston geometries.
The purpose of this paper is to realize volumes of representations as real-valued Chern--Simons invariants in three of these geometries.

Throughout this paper, all manifolds are connected and oriented unless stated otherwise.
Let $G$ be a connected Lie group and $K \subset G$ a closed subgroup such that $X = G/K$ is a contractible $3$-manifold, and fix a $G$-invariant volume form on $X$.
For a closed $3$-manifold $M$, any representation $\rho \colon \pi_1(M) \to G$ then determines the volume $\Vol_G(M,\rho) \in \R$ of the representation (\cite[Notation 2.4]{DLSW} and Section~\ref{subsec:representation_volume}).
In hyperbolic geometry ($G = \Isom^+\Hyp^3 = \PSL_2\C$), the imaginary part of the complex Chern--Simons invariant of a closed hyperbolic $3$-manifold is a constant multiple of the hyperbolic volume \cite{Yoshida,Neumann_Zagier}.
More generally, for a closed $3$-manifold $M$ that is not necessarily hyperbolic, the same relation holds for the volume of any representation $\rho \colon \pi_1(M) \to \PSL_2\C$ whose associated flat principal $G$-bundle $P_\rho$ is trivial (see \cite[Proposition 1.9(2)]{DLW}).
Similarly, in $\SLtilde$-geometry ($G=\Isom_e\SLtilde$), volumes of representations have been related to Chern--Simons invariants (\cite{BrooksGoldman,Khoi} and \cite[Proposition 1.9(1)]{DLW}).

In this paper, we treat $\Nil$-geometry, $\Sol$-geometry, and Euclidean geometry, where the model space $X$ is one of $\Nil$, $\Sol$, and $\E^3$, each of which is a simply connected Lie group acting on itself by left translations.
We take $G = \Isom_e X$, the identity component of the isometry group, which decomposes as the semidirect product
\begin{equation}
    \Isom_e X = X \rtimes K, \qquad K = \SO(2),\ \{1\},\ \SO(3),
    \label{eq:semidirect_decomposition}
\end{equation}
where $K$ is the stabilizer of the identity element of $X$ \cite[\S3]{HaLee}.
Throughout, the structure group is this identity component rather than the full isometry group.
In particular, a geometric structure on $M$ is understood to have holonomy in $\Isom_e X$ (Remark~\ref{rmk:structure_group_connected_components}).
For such a structure, we take the standard metric on $X$ specified in each geometry, denote by $g$ the induced metric on $M$, and orient $M$ via the developing map.
Since $G/K \cong X$ is contractible and carries a $G$-invariant volume form, unique up to a constant multiple and normalized in each of Sections~\ref{sec:Nil}--\ref{sec:Euclid}, the volume of any representation is well-defined.
Chern--Simons forms and invariants are taken with respect to a $G$-invariant polynomial $\p$, that is, an $\Ad(G)$-invariant symmetric bilinear form on the Lie algebra of $G$.
The choice of $\p$ is made separately for the three geometries.

In these three geometries, two difficulties arise.
First, the flat $G$-bundle $P_\rho$ need not be trivial, so that a global section of $P_\rho$, that is, a trivialization, is not available in general.
Second, even when $P_\rho$ is trivial, pulling back the Chern--Simons form of the flat connection $\omega_\rho$ on $P_\rho$ by a global section $s \colon M \to P_\rho$ does not yield the volume of the representation.
In $\Nil$- and $\Sol$-geometry the resulting $3$-form is exact, so that its integral vanishes (Propositions~\ref{prop:flat_and_trivial_connection_cs} and~\ref{prop:cs_triv_flat_sol_exact} together with Remark~\ref{rmk:rel_cs_and_absol_cs}).
In Euclidean geometry, for a suitable normalization of $\p$, its integral defines an $\R/\Z$-valued invariant determined by the $\so(3)$-components of $s^*\omega_\rho$ alone (Theorem~\ref{thm:cs_inv_flat_euclid}).

To handle the first difficulty, we employ the relative Chern--Simons form (\cite{Freed2,DupontJohansen} and Remark~\ref{rmk:rel_cs_and_absol_cs}), which is defined on the base space $M$ for a pair of connections without any trivialization of $P_\rho$.
Its ambiguity under gauge transformations is controlled by the period group $\Lambda_\p(P_\rho) \subset \R$ (Definition~\ref{dfn:gauge_period_P}), and the Chern--Simons invariant of a pair of connections is defined modulo this group (Definition~\ref{dfn:cs_inv_definition}).
We give a vanishing criterion for the period group (Proposition~\ref{prop:period_is_zero} and Corollary~\ref{cor:period_is_zero_criteria}) and verify it in each of the three geometries, so that the invariant is $\R$-valued.

To handle the second difficulty, we pair $\omega_\rho$ with an auxiliary connection on $P_\rho$.
Even though $P_\rho$ may be non-trivial, the fiber bundle $P_\rho/K \to M$ always admits a section because the fiber $G/K \cong X$ is contractible, and from such a section we construct the auxiliary connection by an explicit formula in terms of the components of $\omega_\rho$.
A direct computation using the flatness of $\omega_\rho$ shows that the relative Chern--Simons form of the auxiliary connection and $\omega_\rho$ equals the volume form of the representation for a suitably normalized $G$-invariant polynomial $\p$.

Our main result is as follows.

\begin{maintheorem}[Theorems~\ref{thm:cs_inv_is_vol_nil}, \ref{thm:cs_inv_vol_sol}, and~\ref{thm:euclidean_cs_equals_volume}]
    Let $X$ be one of $\Nil$, $\Sol$, and $\E^3$ with the standard metric and the volume form fixed in each geometry, and set $G=\Isom_e X$.
    Let $M$ be a connected closed oriented $3$-manifold, and let $\rho \colon \pi_1(M) \to G$ be a representation. 
    Then the flat connection $\omega_\rho$ and an auxiliary connection constructed from $\omega_\rho$ determine an $\R$-valued Chern--Simons invariant $\CSvol{X}(M,\rho)$ satisfying
    \[
        \CSvol{X}(M,\rho) = \Vol_G(M,\rho) \in \R.
    \]
\end{maintheorem}

The invariant $\CSvol{X}(M,\rho)$ is defined with respect to the $G$-invariant polynomial $\p$ fixed in each geometry, given by \eqref{eq:inv_poly_nil}, \eqref{eq:inv_poly_sol}, or \eqref{eq:inv_poly_euclid} and normalized as in Sections~\ref{subsec:nil_connection}, \ref{subsec:sol_connection}, and~\ref{subsec:euclid_connection}.
The auxiliary connection is constructed in those sections from a section of $P_\rho/K$.
The value $\CSvol{X}(M,\rho)$ is independent of the choice of the section of $P_\rho/K$ and invariant under conjugation of $\rho$, and no triviality of $P_\rho$ is assumed.
We call $\CSvol{X}(M,\rho)$ the \emph{volume-type Chern--Simons invariant}.
If $\rho_\gm$ is the holonomy representation of a geometric structure on $M$, the volume of the representation coincides with the Riemannian volume, yielding $\CSvol{X}(M,\rho_\gm) = \Vol(M,g)$ in all three geometries (Corollaries~\ref{cor:cs_inv_nil_geohol}, \ref{cor:cs_inv_sol_geohol}, and~\ref{cor:euclidean_cs_vol_geohom}).

We also compute the Chern--Simons invariant $\CS(M,g)$ of the Levi-Civita connection for representative closed manifolds supporting each of the three geometries.
Introduced by Chern and Simons \cite[\S6]{CS1}, this is an $\R/\Z$-valued invariant of a closed Riemannian $3$-manifold $(M,g)$ defined from the orthonormal frame bundle.
To carry out these computations, we establish a formula for $\CS(M,g)$ when $M$ is the quotient of a $3$-dimensional unimodular Lie group by a lattice and $g$ is induced by a left-invariant metric (Lemma~\ref{lem:metric_cs_left_invariant}).
The formula is expressed in terms of the structure constants of the Lie algebra.
With the standard metrics used in this paper, the quotient $M = \Gamma \backslash \Nil$ by a lattice $\Gamma \subset \Nil$ satisfies
\[
    \CS(M,g) \equiv - \frac{1}{16\pi^2} \Vol(M,g) \pmod{\Z}
\]
by Proposition~\ref{prop:metric_cs_nil}.
In contrast, the invariant vanishes in $\Sol$-geometry whenever the image of the holonomy representation is a lattice in $\Sol$ (Proposition~\ref{prop:metric_cs_sol}).
It also vanishes in Euclidean geometry for the flat metrics treated in Proposition~\ref{prop:metric_cs_euclid}, whose proof uses known computations of the eta invariant.

In Euclidean geometry, moreover, the Chern--Simons invariant of the Levi-Civita connection is itself recovered within our framework.
For the holonomy representation $\rho_\gm$, the flat $\SE(3)$-bundle $P_{\rho_\gm}$ is trivial as a principal bundle, and the invariant polynomial is described by two parameters $(c_1,c_2) \in \R^2$, as in \eqref{eq:inv_poly_euclid}.
Whether the resulting Chern--Simons invariant takes values in $\R/\Z$ or in $\R$ depends on the choice of the polynomial and of the pair of connections.
Corresponding to $c_1$, the flat connection and a trivial connection determine an $\R/\Z$-valued invariant, the rotation-type invariant $\CSrot{\E^3}(M,\rho_\gm)$, which agrees with $\CS(M,g)$ (Corollary~\ref{cor:cs_inv_flat_euclid_geometryhol}).
Corresponding to $c_2$, the auxiliary connection and the flat connection determine the $\R$-valued volume-type invariant $\CSvol{\E^3}(M,\rho_\gm)$ of Theorem~\ref{thm:euclidean_cs_equals_volume}.

\emph{Conventions and organization.}
From Section~\ref{sec:cs_inv_and_rep_volume} onwards, we write \CSname{} for Chern--Simons.
Throughout this paper, all manifolds and maps between manifolds are smooth.
Unless stated otherwise, manifolds are connected and oriented, $M$ denotes a closed $3$-manifold, and $G$ denotes a connected real Lie group.
All bundle sections are global.
In each section from Section~\ref{sec:Nil} onwards, $X$ denotes the model space, $G = \Isom_e X$, and $K$ the stabilizer appearing in \eqref{eq:semidirect_decomposition}.
Section~\ref{sec:cs_inv_and_rep_volume} reviews the relative Chern--Simons form, the Chern--Simons invariant, the volume of representations, and the Chern--Simons invariant of the Levi-Civita connection.
Sections~\ref{sec:Nil}, \ref{sec:Sol}, and~\ref{sec:Euclid} treat $\Nil$-geometry, $\Sol$-geometry, and Euclidean geometry, respectively.

\section{Preliminaries: Chern--Simons invariants and volumes of representations}
\label{sec:cs_inv_and_rep_volume}

In this section, we review the \CSname{} invariant and the volume of representations used throughout this paper.

\subsection{Chern--Simons invariants}
\label{subsec:CSinv}

Following \cite{CS1,Freed2,DupontJohansen}, we briefly review the relative \CSname{} form and its properties.
Throughout this section, $\pi \colon P \to M$ denotes a principal $G$-bundle.
We write $\mathcal{G}(P)$ for the group of gauge transformations of $P$.
Throughout, $\g$ and $\kk$ denote the Lie algebras of $G$ and of a closed subgroup $K\subset G$, and $\mathcal{A}^k(X;E)$ denotes the space of $E$-valued $k$-forms on a manifold $X$.
We write $\ad P\coloneqq P\times_{\Ad}\g$ for the adjoint bundle of $P$, and $B^*$ for the fundamental vector field on $P$ generated by $B\in\g$.

Given two connections $\omega_1$ and $\omega_0$ on $P$, set $a \coloneqq \omega_1 - \omega_0 \in \mathcal{A}^1(M; \ad P)$.
For a $G$-invariant polynomial $\p=\langle \cdot, \cdot \rangle$ of degree $2$, that is, an $\Ad(G)$-invariant symmetric bilinear form, the relative \CSname{} $3$-form is given by
\begin{equation}
	\CS_\p(\omega_1, \omega_0)
	= \langle a, \diff_{\omega_0} a \rangle + \frac{1}{3} \langle a, [a, a] \rangle + 2 \langle a, \Omega_0 \rangle \in \mathcal{A}^3(M;\R) \, .
	 \label{eq:cs_formula_general_l2}
\end{equation}
Here $\diff_{\omega_0} a \coloneqq \diff a + [\omega_0, a]$ is the covariant derivative with respect to $\omega_0$, and $\Omega_i \coloneqq \diff\omega_i+(1/2)[\omega_i,\omega_i]$ is the curvature of $\omega_i$.
The bracket of $\g$ and the bilinear form $\p$ are extended to $\g$-valued forms by combining them with the exterior product, so that $[\alpha\otimes\xi,\beta\otimes\eta]=(\alpha\wedge\beta)\otimes[\xi,\eta]$ and $\langle\alpha\otimes\xi,\beta\otimes\eta\rangle=(\alpha\wedge\beta)\langle\xi,\eta\rangle$; in particular $[a,a](u,v)=2[a(u),a(v)]$.
In particular, when $\omega_0$ is flat, this simplifies to
\begin{equation}
	\CS_\p(\omega_1, \omega_0) = \langle a, \diff_{\omega_0} a \rangle + \frac{1}{3} \langle a, [a, a] \rangle \in \mathcal{A}^3(M;\R) \, . \label{eq:cs_formula_flat_l2}
\end{equation}
If $\omega_1$ is also flat, the relation $\Omega_1 = \Omega_0 + \diff_{\omega_0}a + (1/2)[a,a]$ gives the Maurer--Cartan equation $\diff_{\omega_0}a + (1/2)[a,a] = 0$.
Substituting this into \eqref{eq:cs_formula_flat_l2}, we obtain
\begin{equation}
	\CS_\p(\omega_1, \omega_0) = -\frac{1}{6} \langle a, [a, a] \rangle \in \mathcal{A}^3(M;\R).
	\label{eq:cs_formula_flat_flat}
\end{equation}

We collect several properties of the \CSname{} form.
Let $\omega_0$, $\omega_1$, and $\omega_2$ be arbitrary connections on $P$.
\begin{enumerate}
	\item Antisymmetry {\cite[Equation 1.7]{DupontJohansen}}:
	      \begin{equation}
		      \CS_\p(\omega_1, \omega_0) = - \CS_\p(\omega_0, \omega_1) \in \mathcal{A}^{3}(M;\R). \label{eq:cs_antisymmetry}
	      \end{equation}
	\item Cocycle condition {\cite[Equation 1.14]{DupontJohansen}}:
	      \begin{equation}
		      \CS_\p(\omega_1, \omega_0) + \CS_\p(\omega_2, \omega_1) - \CS_\p(\omega_2, \omega_0) = 
			  \diff \langle \omega_1 - \omega_0,\ \omega_2 - \omega_1 \rangle \in \mathcal{A}^{3}(M;\R). \label{eq:cs_transitivity}
	      \end{equation}
	\item Naturality {\cite[Proposition 1.27(d)]{Freed1}}:
	      let $P' \to M'$ be a principal $G$-bundle with connections $\omega_0$ and $\omega_1$.
	      For any bundle map $F \colon P \to P'$ covering $f \colon M \to M'$, we have
	      \begin{equation}
		      \CS_\p(F^*\omega_1, F^*\omega_0) = f^*\CS_\p(\omega_1, \omega_0) \in \mathcal{A}^{3}(M;\R).
			  \label{eq:cs_naturality}
	      \end{equation}
	\item Diagonal gauge invariance: for every gauge transformation $\varphi \in \mathcal{G}(P)$,
	      \begin{equation}
		      \CS_\p(\varphi^*\omega_1, \varphi^*\omega_0) = \CS_\p(\omega_1, \omega_0) \in \mathcal{A}^{3}(M;\R). \label{eq:cs_diagonal_gauge_invariance}
	      \end{equation}
		  This is the special case of \eqref{eq:cs_naturality} in which the bundle map is a gauge transformation.
	\item Gauge transformation law {\cite[Proposition 1.27(e)]{Freed1}}: modulo exact forms on $P$, every gauge transformation $\varphi \in \mathcal{G}(P)$ satisfies
	     \begin{equation}
		    \pi^*\CS_\p(\varphi^*\omega_1, \omega_1) \equiv \pi^*\CS_\p(\varphi^*\omega_0, \omega_0)
			\equiv F_\varphi^* \WZ_\p \in \mathcal{A}^{3}(P;\R). \label{eq:cs_gauge_on_P}
	     \end{equation}
		Here $F_\varphi \colon P \to G$ is the map associated with $\varphi$, determined by $\varphi(u) = u \cdot F_\varphi(u)$ for $u \in P$; it satisfies $F_\varphi(ug) = g^{-1}F_\varphi(u)g$.
		We write $\omega_\MC$ for the Maurer--Cartan form of $G$ and define the Wess--Zumino form by
		\[
			\WZ_\p \coloneqq -\frac{1}{6}\langle \omega_\MC, [\omega_\MC,\omega_\MC] \rangle \in \mathcal{A}^3(G;\R).
		\]
	\end{enumerate}

\begin{rmk}
	\label{rmk:rel_cs_and_absol_cs}
	The (absolute) \CSname{} form of a connection $\omega$ \cite{CS1,Freed1} is defined on the total space $P$ by
	\[
		\CS_\p(\omega) \coloneqq \langle \omega , \diff \omega \rangle + \frac{1}{3} \langle \omega, [\omega,\omega] \rangle \in \mathcal{A}^3(P;\R).
	\]
	In contrast, the relative \CSname{} form is defined on the base space $M$.
	For any two connections $\omega_0$ and $\omega_1$, a direct computation shows that the two forms are related by
	\begin{equation}
		\CS_\p(\omega_1) - \CS_\p(\omega_0) = \pi^*\CS_\p(\omega_1, \omega_0) - \diff \langle \omega_0,\ \omega_1 - \omega_0 \rangle \in \mathcal{A}^3(P;\R).
		\label{eq:rel_cs_and_absol_cs}
	\end{equation}
	The properties \eqref{eq:cs_naturality}, \eqref{eq:cs_diagonal_gauge_invariance}, and \eqref{eq:cs_gauge_on_P} collected above are stated in the references for $\CS_\p(\omega)$, but they hold for the relative \CSname{} form as well:
	naturality \eqref{eq:cs_naturality}, and its special case \eqref{eq:cs_diagonal_gauge_invariance}, follow because each term of the defining formula \eqref{eq:cs_formula_general_l2} is natural with respect to bundle maps, and \eqref{eq:cs_gauge_on_P} follows from \eqref{eq:rel_cs_and_absol_cs}.
	In particular, when $P$ is trivial and $\omega_\triv^s$ is the trivial connection for a section $s \colon M \to P$, 
	pulling back \eqref{eq:rel_cs_and_absol_cs} by $s$ yields 
	\[
		\CS_\p(\omega, \omega_\triv^s) = s^*\CS_\p(\omega) \in \mathcal{A}^3(M;\R).
	\]
\end{rmk}

Unless stated otherwise, $\rho$ denotes an arbitrary representation $\pi_1(M) \to G$, with corresponding flat $G$-bundle $P_\rho \coloneqq \widetilde{M} \times_\rho G \to M$ and flat connection $\omega_\rho$, where $\widetilde{M}$ is the universal cover of $M$.
When $P_\rho$ is trivial, we write $\omega_\triv^s$ for the trivial connection associated with a section $s \colon M \to P_\rho$.
In each geometry, we pair the flat connection $\omega_\rho$ with an auxiliary connection constructed to produce the volume form of the representation.
To obtain an invariant that is independent of gauge transformations on either connection, we introduce the \CSname{} period group below.

\begin{dfn}
	\label{dfn:gauge_period_P}
	We define the period group $\Lambda_\p(P)$ of a principal $G$-bundle $P$ by
	\begin{equation}
		\Lambda_\p(P) \coloneqq \left \{ \left. \int_M \CS_\p(\varphi^*\omega, \omega) \,\in \R \right| \varphi \in \mathcal{G}(P) \right\} \subset \R . \label{eq:gauge_period_P}
	\end{equation}
	By Remark~\ref{rmk:period_well_defined}, this is a subgroup of $\R$ and is independent of the choice of the connection $\omega$.
\end{dfn}

\begin{rmk}
	\label{rmk:period_well_defined}
	Since $M$ is closed, the cocycle condition \eqref{eq:cs_transitivity} and Stokes' theorem give, for arbitrary connections $\omega_0$, $\omega_1$, and $\omega_2$,
	\begin{equation}
		\int_M \CS_\p(\omega_1,\omega_0) + \int_M \CS_\p(\omega_2,\omega_1) = \int_M \CS_\p(\omega_2,\omega_0).
		\label{eq:integral_cocycle}
	\end{equation}
	Applying this repeatedly to $\varphi^*\omega_1$, $\varphi^*\omega_0$, $\omega_0$, and $\omega_1$, and using antisymmetry \eqref{eq:cs_antisymmetry} together with diagonal gauge invariance \eqref{eq:cs_diagonal_gauge_invariance}, we obtain
	\[
		\int_M \CS_\p(\varphi^*\omega_1,\omega_1) = \int_M \CS_\p(\varphi^*\omega_0,\omega_0) ,
	\]
	so that $\Lambda_\p(P)$ does not depend on the choice of the connection.
	Applying \eqref{eq:integral_cocycle} and this independence to $(\varphi\psi)^*\omega = \psi^*(\varphi^*\omega)$ shows that $\varphi \mapsto \int_M \CS_\p(\varphi^*\omega,\omega)$ is a group homomorphism $\mathcal{G}(P) \to \R$, and hence $\Lambda_\p(P)$, being its image, is a subgroup of $\R$.
\end{rmk}

We now use the period group to define the \CSname{} invariant associated with two connections $\omega_0$ and $\omega_1$.
First, by \eqref{eq:integral_cocycle} and antisymmetry \eqref{eq:cs_antisymmetry}, every gauge transformation $\varphi \in \mathcal{G}(P)$ satisfies
\begin{align*}
	\int_M \CS_\p(\omega_1, \varphi^*\omega_0) &= \int_M \CS_\p(\omega_1,\omega_0) - \int_M \CS_\p(\varphi^*\omega_0,\omega_0) , \\
	\int_M \CS_\p(\varphi^*\omega_1,\omega_0) &= \int_M \CS_\p(\omega_1,\omega_0) + \int_M \CS_\p(\varphi^*\omega_1,\omega_1) .
\end{align*}
In both cases the second term on the right-hand side lies in $\Lambda_\p(P)$, so that
\begin{align}
	\int_M \CS_\p(\omega_1,\omega_0)
	&\equiv \int_M \CS_\p(\omega_1, \varphi^*\omega_0) \equiv \int_M \CS_\p(\varphi^*\omega_1,\omega_0)  \pmod{\Lambda_\p(P)}.
	\label{eq:cs_inv_gauge_indep}
\end{align}
Modulo $\Lambda_\p(P)$, the integral of the \CSname{} form is thus unchanged under a gauge transformation of either of the two connections.
\begin{dfn}
	\label{dfn:cs_inv_definition}
	We define the (relative) \CSname{} invariant associated with connections $\omega_0$ and $\omega_1$ on a principal $G$-bundle $P$ by
	\begin{equation}
		\CS(M;\omega_1,\omega_0) \coloneqq \int_M \CS_\p(\omega_1,\omega_0) \pmod{\Lambda_\p(P)} .
		\label{eq:cs_inv_definition}
	\end{equation}
	By \eqref{eq:cs_inv_gauge_indep}, this is invariant under independent gauge transformations of the two connections.
\end{dfn}

The \CSname{} invariant $\CS(M;\omega_1,\omega_0)$ of Definition~\ref{dfn:cs_inv_definition} takes values in $\R/\Lambda_\p(P)$.
Directly computing $\Lambda_\p(P)$ is difficult because it depends on the gauge group $\mathcal{G}(P)$.
We provide two ways to control $\Lambda_\p(P)$:
First, when $P$ is trivial, we bound $\Lambda_\p(P)$ by a subgroup $\Lambda_\p(G)$ that depends only on the topology of $G$ (see \eqref{eq:universal_period} and Lemma~\ref{lem:expanded_period}).
This bound yields the $\R/\Z$-valued invariants in Sections~\ref{subsec:metric_cs} and~\ref{subsec:euclid_flat}.
Second, for flat bundles $P_\rho$, we give a criterion ensuring $\Lambda_\p(P_\rho) = 0$ (Proposition~\ref{prop:period_is_zero} and Corollary~\ref{cor:period_is_zero_criteria}), which allows us to formulate all main theorems with $\R$-valued invariants.

We first bound $\Lambda_\p(P)$ from above in the case of a trivial bundle.
For a connected real Lie group $G$ and an invariant polynomial $\p$, we define the period group $\Lambda_\p(G)$ of $G$ by
\begin{equation}
	\Lambda_\p(G) \coloneqq \left\{ \left. \int_{\gamma} \WZ_\p \in \R \,\right|\, [\gamma] \in H_{3}(G;\Z) \right\} \subset \R .
	\label{eq:universal_period}
\end{equation}
Since $[\gamma] \mapsto \int_\gamma \WZ_\p$ is a group homomorphism, $\Lambda_\p(G)$ is a subgroup of $\R$.

\begin{lem}
	\label{lem:expanded_period}
	Let $P$ be a trivial $G$-bundle.
	Then $\Lambda_\p(P) \subset \Lambda_\p(G)$.
	In particular, the \CSname{} invariant \eqref{eq:cs_inv_definition} is defined as an element of $\R/\Lambda_\p(G)$.
\end{lem}

\begin{proof}
	Fix the section $s \colon M \to P$ determined by a trivialization $P \cong M \times G$.
	Pulling back the gauge transformation law \eqref{eq:cs_gauge_on_P} by $s$ and integrating over the closed manifold $M$, we obtain from Stokes' theorem
	\[
		\int_M \CS_\p(\varphi^*\omega,\omega) = \int_M f^*\WZ_\p = \int_{f_*[M]} \WZ_\p ,
	\]
	where $[M] \in H_{3}(M;\Z)$ is the fundamental class of $M$ and $f \coloneqq F_\varphi \circ s$.
	The right-hand side lies in $\Lambda_\p(G)$, and $\Lambda_\p(P) \subset \Lambda_\p(G)$ follows.
\end{proof}

Next, we give a sufficient condition for $\Lambda_\p(P_\rho) = 0$ in the case of a flat bundle.
From now on, let $K$ be a maximal compact subgroup of $G$.
Since $G$ is connected, the Cartan--Iwasawa--Malcev theorem shows that $G$ is diffeomorphic to $K \times \R^n$; in particular, the homogeneous space $G/K \cong \R^n$ is contractible.
The $G/K$-fiber bundle $P_\rho/K \to M$ therefore admits a section, and any two sections are homotopic.
For a section $\sigma \colon M \to P_\rho/K$, we write $Q_\sigma$ for the reduced sub-$K$-bundle, that is, the preimage of $\sigma(M)$ under the projection $P_\rho \to P_\rho/K$.
We write $\iota_\sigma \colon Q_\sigma \hookrightarrow P_\rho$ for its inclusion into $P_\rho$, and $\pi_\sigma \coloneqq \pi\circ\iota_\sigma \colon Q_\sigma \to M$ for the projection onto the base space.
We will use this contractibility and notation repeatedly throughout the paper.

\begin{prop}
	\label{prop:period_is_zero}
	If $[\WZ_\p] = 0$ in $H^{3}(G;\R)$ and the pullback map $\pi^* \colon H^{3}(M;\R) \to H^{3}(P_\rho;\R)$ is injective, then $\Lambda_\p(P_\rho) = 0$.
	In particular, the \CSname{} invariant takes values in $\R$.
\end{prop}

\begin{proof}
	By the gauge transformation law \eqref{eq:cs_gauge_on_P} on the total space, every gauge transformation $\varphi \in \mathcal{G}(P_\rho)$ satisfies $\pi^*\CS_\p(\varphi^*\omega,\omega) \equiv F_\varphi^*\WZ_\p$ modulo exact forms on $P_\rho$.
	By assumption $\WZ_\p$ is exact, hence so is its pullback $F_\varphi^*\WZ_\p$, and therefore $\pi^*\CS_\p(\varphi^*\omega,\omega)$ is exact on the total space.
	Since $\pi^*$ is assumed to be injective on cohomology, $\CS_\p(\varphi^*\omega,\omega)$ is exact on the base space $M$ as well.
	As $M$ is closed, Stokes' theorem gives $\int_M \CS_\p(\varphi^*\omega,\omega) = 0$.
	Since $\varphi$ is arbitrary, the definition \eqref{eq:gauge_period_P} of $\Lambda_\p(P_\rho)$ gives $\Lambda_\p(P_\rho) = 0$.
\end{proof}

Of the two hypotheses of Proposition~\ref{prop:period_is_zero}, the injectivity of $\pi^*$ is hard to verify directly.
In practice we use the following three cases, in which the injectivity is verified through a section or through the reduction to a maximal compact subgroup $K$.

\begin{cor}
	\label{cor:period_is_zero_criteria}
	In the notation of Proposition~\ref{prop:period_is_zero}, we have $\Lambda_\p(P_\rho) = 0$ if one of the following holds:
	\begin{enumerate}
	\item $K \cong \SO(2)$, and the real Euler class of the principal $S^1$-bundle obtained by reducing the structure group of $P_\rho$ to $K$ vanishes in $H^2(M;\R)$.
	\label{item:period_zero_euler}
	\item $[\WZ_\p] = 0 \in H^{3}(G;\R)$ and $P_\rho$ is a trivial $G$-bundle.
	\label{item:period_zero_trivial}
	\item $[\WZ_\p] = 0 \in H^{3}(G;\R)$ and $K$ satisfies $H^1(K;\R) = H^2(K;\R) = 0$.
	\label{item:period_zero_cohomology}
	\end{enumerate}
\end{cor}

\begin{proof}
	In each case it suffices to verify the two hypotheses of Proposition~\ref{prop:period_is_zero}.

	\emph{Case~\ref{item:period_zero_euler}.}
	Since $K\cong\SO(2)$ and $G/K$ is contractible, the inclusion $K\hookrightarrow G$ is a homotopy equivalence, so that $H^3(G;\R) \cong H^3(S^1;\R) = 0$ and $[\WZ_\p]=0$ holds automatically.
	In the Gysin sequence
	\[
		H^1(M;\R)\xrightarrow{\cupdot e}H^3(M;\R)\xrightarrow{\pi_\sigma^*}H^3(Q_\sigma;\R)
	\]
	of the principal $S^1$-bundle $\pi_\sigma \colon Q_\sigma \to M$, the real Euler class satisfies $e=0$ by assumption, so that $\pi_\sigma^* = \iota_\sigma^*\circ\pi^*$ is injective; hence so is $\pi^*$.

	\emph{Case~\ref{item:period_zero_trivial}.}
	As $P_\rho$ is a trivial $G$-bundle, it admits a section $s \colon M \to P_\rho$, and $\pi\circ s = \id_M$ gives $s^* \circ \pi^* = \id$ on cohomology.
	In particular, $\pi^*$ is injective.

	\emph{Case~\ref{item:period_zero_cohomology}.}
	Take again the reduced sub-$K$-bundle $\pi_\sigma \colon Q_\sigma \to M$ and consider its Serre spectral sequence $E_2^{p,q} = H^p(M;H^q(K;\R))$.
	Since $K$ is connected, the action of $\pi_1(M)$ on $H^*(K;\R)$ is trivial and the local system is trivial.
	The differentials $d_r \colon E_r^{3-r,r-1} \to E_r^{3,0}$ into $E^{3,0}$ all vanish: for $r=2$ the source is $H^1(M;H^1(K;\R)) = 0$; for $r=3$ it is a subquotient of $H^0(M;H^2(K;\R)) = 0$; and for $r \geq 4$ the first index $3-r$ is negative.
	The differentials out of $E^{3,0}$ vanish as well, since their second index $1-r$ is negative, and therefore $E_\infty^{3,0} = E_2^{3,0} = H^3(M;\R)$.
	Hence $\pi_\sigma^*$ is injective, and so is $\pi^*$.
\end{proof}

\subsection{Volume of representations}
\label{subsec:representation_volume}

In this subsection, we define the volume of a representation, written $\Vol_G(M,\rho)$.
We assume in addition that $\dim G/K = 3$, and we fix a $G$-invariant volume form $\vol_G \in \mathcal{A}^3(G/K)^G$ on $G/K$.
Since $\vol_G$ is $G$-invariant, the $3$-form $\mathrm{pr}_2^*\vol_G$ on $\widetilde{M} \times G/K$ is invariant under the action of $\pi_1(M)$.
It therefore defines a closed $3$-form $\vol_\rho \in \mathcal{A}^3(P_\rho/K;\R)$ on the quotient bundle $P_\rho/K$.
Its restriction to each fiber agrees with $\vol_G$ under the identification with $G/K$.

For an arbitrary representation $\rho \colon \pi_1(M) \to G$, we define the volume $\Vol_G(M, \rho)$ of the representation by
\begin{equation}
	\Vol_G(M, \rho) \coloneqq \int_M \sigma^*\vol_\rho \;\in\; \R ,
	\label{eq:representation_volume}
\end{equation}
where $\sigma \colon M \to P_\rho/K$ is an arbitrary section.
This definition coincides with the volume of representation in \cite[Notation 2.4]{DLSW}.
It is independent of the choice of section \cite[Lemma 2.2(1)]{DLSW} and invariant under conjugation of $\rho$ \cite[Proposition 2.3]{DLSW}.
We call $\sigma^*\vol_\rho$ the volume form of the representation with respect to $\sigma$.

\begin{rmk}
	\label{rmk:structure_group_connected_components}
	In this paper, we take the identity component of the isometry group as the structure group, so that $G$ is one of $\Isom_e\Nil$, $\Isom_e\Sol$, and $\Isom_e\E^3$.
	We restrict to the identity component because our argument for the contractibility of $G/K$ requires $G$ to be connected, and the orientation of $M$ via the developing map requires orientation-preserving isometries.
	Manifolds carrying a geometric structure whose holonomy does not lie in $\Isom_e X$ therefore fall outside our scope.
	For these three groups $G$, the stabilizer $K$ in \eqref{eq:semidirect_decomposition} is a maximal compact subgroup of $G$, and the homogeneous space $G/K$ is $\Nil$, $\Sol$, and $\E^3$, respectively; each of them is contractible, and a $G$-invariant volume form exists and is unique up to a constant multiple (see also \cite[Lemma 2.1(2)]{DLSW}).
	Fixing a $G$-invariant volume form in each geometry thus determines the volume \eqref{eq:representation_volume} of a representation.
\end{rmk}

We collect the properties of the volume of representations.
For $(X,G)$-structures and their developing maps and holonomy representations, we follow \cite{Thurston}.

\begin{lem}
	\label{lem:rep_vol_properties}
	\begin{enumerate}
	\item
	\label{item:rep_vol_naturality}
	Let $N$ be a closed $3$-manifold and let $f \colon M \to N$ be a smooth map.
	Fix base points, and let $f_* \colon \pi_1(M) \to \pi_1(N)$ be the induced homomorphism.
	Then every representation $\rho \colon \pi_1(N) \to G$ satisfies
	\begin{equation}
		\Vol_G(M, \rho\circ f_*) = \deg(f)\cdot\Vol_G(N, \rho).
		\label{eq:rep_vol_naturality}
	\end{equation}
	\item
	\label{item:rep_vol_geometric_holonomy}
	Suppose that $G$ acts transitively on a manifold $X$ and that the stabilizer of a chosen point of $X$ is $K$, so that $X \cong G/K$.
	Let $\vol_G$ be the volume form of a $G$-invariant Riemannian metric on $X$.
	Suppose that $M$ admits an $(X,G)$-structure with holonomy representation $\rho_\gm \colon \pi_1(M) \to G$.
	Let $g$ be the metric on $M$ induced by the developing map $D \colon \widetilde{M} \to X$, and let $M$ be oriented via $D$.
	Then
	\begin{equation}
		\Vol_G(M,\rho_\gm) = \Vol(M,g),
		\label{eq:rep_vol_geometric_holonomy}
	\end{equation}
	where $\Vol(M,g)$ is the volume of the Riemannian manifold $(M,g)$.
	\end{enumerate}
\end{lem}

\begin{proof}
	Part~\ref{item:rep_vol_naturality} is \cite[Proposition 3.1(1)]{DLSW}.
	For part~\ref{item:rep_vol_geometric_holonomy}, the equivariance of $D$ with respect to $\rho_\gm$ implies that $[\widetilde{m}] \mapsto [\widetilde{m}, D(\widetilde{m})]$ defines a section $\sigma_D \colon M \to P_{\rho_\gm}/K$.
	Since $G$ is connected, it acts on $X$ preserving the orientation, and the pullback of the orientation of $X$ under $D$ is a well-defined orientation of $M$.
	Hence $D$ is an orientation-preserving local isometry, and $\sigma_D^*\vol_\rho$ agrees with the Riemannian volume form of $g$; the assertion follows from \eqref{eq:representation_volume}.
\end{proof}

\subsection{The Chern--Simons invariant of the Levi-Civita connection}
\label{subsec:metric_cs}

In this subsection, we review the \CSname{} invariant $\CS(M,g)$ of the Levi-Civita connection determined by a Riemannian metric $g$.
This invariant was introduced in \cite{CS1}; note that it is a quantity on the frame bundle of $(M,g)$ rather than on the flat $G$-bundle $P_\rho$.

We write $\FrameBundle$ for the orthonormal frame bundle of a closed Riemannian $3$-manifold $(M,g)$ and $\omega_\LC$ for its Levi-Civita connection.
As is well known, $M$ is parallelizable and $\FrameBundle$ is a trivial $\SO(3)$-bundle.
Choosing an orthonormal frame $s \colon M \to \FrameBundle$, we define the \CSname{} invariant $\CS(M,g)$ of the Levi-Civita connection by
\begin{equation}
	\CS(M,g) \coloneqq
	-\frac{1}{16\pi^2}\int_M s^*\trace\left(\omega_\LC \wedge \diff\omega_\LC + \frac{2}{3}\omega_\LC \wedge \omega_\LC \wedge \omega_\LC\right) \pmod{\Z} .
	\label{eq:metric_cs_explicit}
\end{equation}
This is precisely the invariant $\int_M (1/2)\,s^*TP_1(\omega_\LC)$ given in \cite[\S6]{CS1}.
As we see in Remark~\ref{rmk:metric_cs_as_relative}, it does not depend on the choice of the frame $s$.

\begin{rmk}
	\label{rmk:metric_cs_as_relative}
	The \CSname{} invariant \eqref{eq:metric_cs_explicit} of the Levi-Civita connection can also be written in terms of the relative \CSname{} form of this paper.
	As the invariant polynomial we take half of the first Pontryagin polynomial $P_1(X,Y) = -\frac{1}{8\pi^2}\trace(XY)$ of \cite[\S6]{CS1}, that is, $\p \coloneqq (1/2)P_1$, and we write $\omega_\triv^s$ for the trivial connection associated with the frame $s$.
	Setting $A \coloneqq s^*\omega_\LC$, we have $\CS_\p(\omega_\LC,\omega_\triv^s) = \langle A,\diff A \rangle + (1/3)\langle A,[A,A] \rangle$ by \eqref{eq:cs_formula_flat_l2}, and therefore
	\[
		\CS(M,g) \equiv \CS(M;\omega_\LC,\omega_\triv^s) \equiv \int_M \CS_\p(\omega_\LC,\omega_\triv^s) \pmod{\Z}.
	\]
	Here $\Lambda_{P_1}(\SO(3)) = 2\Z$, because the standard basis of $\so(3)$ satisfies $\trace(L_iL_j) = -2\delta_{ij}$, and because $\Vol(\SO(3)) = 8\pi^2$ with respect to the metric making this basis orthonormal \cite[Equation 5.13]{CS1}; we take half of $P_1$ in order to normalize this to $\Lambda_\p(\SO(3)) = \Z$.
	Since $\FrameBundle$ is trivial, Lemma~\ref{lem:expanded_period} gives $\Lambda_\p(\FrameBundle) \subset \Z$, and the \CSname{} invariant of Definition~\ref{dfn:cs_inv_definition} is defined modulo $\Z$ and independent of the choice of the frame $s$.
\end{rmk}

In all our examples below, the metric is induced by a left-invariant metric on a Lie group.
We now express $\CS(M,g)$ in terms of the structure constants.

\begin{lem}
	\label{lem:metric_cs_left_invariant}
	Let $X$ be an oriented $3$-dimensional Lie group, let $g$ be a left-invariant metric, let $\Gamma\subset X$ be a discrete cocompact subgroup acting freely on $X$, and equip $M\coloneqq\Gamma\backslash X$ with the orientation induced from $X$.
	Here $(i,j,k)$ denotes a cyclic permutation of $(1,2,3)$, that is, one of $(1,2,3)$, $(2,3,1)$, and $(3,1,2)$.
	Suppose that $\Lie(X)$ admits a positively oriented orthonormal basis $(E_1,E_2,E_3)$ with $[E_i,E_j]=\lambda_k E_k$, and set $\mu_i\coloneqq(\lambda_1+\lambda_2+\lambda_3)/2-\lambda_i$.
	Then
	\begin{equation}
		\CS(M,g) \equiv -\frac{\lambda_1\lambda_2\lambda_3-4\mu_1\mu_2\mu_3}{8\pi^2} \, \Vol(M,g) \pmod{\Z}.
		\label{eq:metric_cs_left_invariant}
	\end{equation}
\end{lem}

\begin{proof}
	The left-invariant frame $(E_1,E_2,E_3)$ determines a positively oriented orthonormal frame $s$ on $M$, and the dual coframe $e^i \in \mathcal{A}^1(M;\R)$ satisfies $\vol_g=e^1\wedge e^2\wedge e^3$ and $\diff e^i=-\lambda_i\,e^j\wedge e^k$.
	The Koszul formula gives $\nabla_{E_i}E_j=\mu_iE_k$, so that, writing $A\coloneqq s^*\omega_\LC$ as a $3\times3$ matrix in $\so(3)$ with entries $\omega_{ab}$ determined by $\nabla E_b=\sum_a \omega_{ab}\,E_a$, we have
	\[
		\omega_{23}=-\mu_1\,e^1,\quad \omega_{31}=-\mu_2\,e^2,\quad \omega_{12}=-\mu_3\,e^3 .
	\]
	The skew-symmetry of $A$ and the expressions above yield
	\begin{align*}
		\trace(A\wedge\diff A+\frac{2}{3}A\wedge A\wedge A)
		&= -2\sum_{a<b}\omega_{ab}\wedge\diff\omega_{ab} + 4\,\omega_{12}\wedge\omega_{23}\wedge\omega_{31} \\
		&= -2\sum_i \mu_i^2\,e^i\wedge\diff e^i - 4\mu_1\mu_2\mu_3\,\vol_g \\
		&= (2\sum_i\mu_i^2\lambda_i-4\mu_1\mu_2\mu_3)\vol_g \\
		&= 2(\lambda_1\lambda_2\lambda_3-4\mu_1\mu_2\mu_3)\,\vol_g ,
	\end{align*}
	where the last equality follows from $\lambda_i=\mu_j+\mu_k$ and $\sum_i\mu_i^2\lambda_i=\lambda_1\lambda_2\lambda_3-2\mu_1\mu_2\mu_3$.
	Substituting this into \eqref{eq:metric_cs_explicit} gives \eqref{eq:metric_cs_left_invariant}.
\end{proof}

\begin{rmk}
	\label{rmk:unimodular_case}
	The three geometries $\Nil$, $\Sol$, and $\E^3$ treated in this paper are all $3$-dimensional unimodular Lie groups, and a unimodular Lie group always admits an orthonormal basis satisfying the hypothesis of Lemma~\ref{lem:metric_cs_left_invariant} \cite[pp.~305--307]{Milnor}; the lemma therefore applies to each of them.
\end{rmk}

Finally, we record without proof the relation between $\CS(M,g)$ and the eta invariant.
The Atiyah--Patodi--Singer index theorem yields the relation
\begin{equation}
	\CS(M,g) \equiv \frac{3}{2}\eta(M,g) + \frac{1}{2}\sigma_2(M) \pmod{\Z}
	\label{eq:aps_relation}
\end{equation}
for a closed Riemannian $3$-manifold \cite[Theorem 4.1]{Ouyang}; the original source is \cite[Proposition 4.19]{APS2}.
Here $\eta(M,g)$ is the eta invariant of the signature operator, and $\sigma_2(M)$ is the number of $2$-primary summands in $H_1(M;\Z)$.
We use \eqref{eq:aps_relation} in Section~\ref{subsec:euclid_examples} to determine $\CS(M,g)$ for flat manifolds.

\section{Chern--Simons invariants in \texorpdfstring{$\Nil$}{Nil}-geometry}
\label{sec:Nil}

In this section, we show that the volume-type \CSname{} invariant in $\Nil$-geometry equals the volume of the representation (Theorem~\ref{thm:cs_inv_is_vol_nil}).
We then examine its relation to the \CSname{} invariant of the Levi-Civita connection.

\subsection{Preliminaries on \texorpdfstring{$\Nil$}{Nil}-geometry}
\label{subsec:nil_prep}

Identify $\Nil$ with $\R^2\times\R$, assigning to $(x,y,z)$ the upper triangular matrix
\begin{equation}
    \begin{pmatrix} 1 & x & z \\ 0 & 1 & y \\ 0 & 0 & 1 \end{pmatrix}
    \label{eq:nil_matrix}
\end{equation}
as in \cite[\S4]{Scott}.
The multiplication is then matrix multiplication, so that $\Nil$ is the Heisenberg group.
We write $X,Y,Z$ for the left-invariant frame in these coordinates, and call the left-invariant metric making them orthonormal the \emph{standard metric} on $\Nil$; in coordinates,
\begin{equation}
    g_{\Nil} = \diff x^2+\diff y^2+(\diff z-x\,\diff y)^2 .
    \label{eq:nil_metric}
\end{equation}

Throughout this section, $G=\Isom_e\Nil \cong \Nil \rtimes \SO(2)$ and $K=\SO(2)$ \cite[Theorem 3.2]{HaLee}.
Each $A\in K$ acts on $\Nil$ isometrically as an automorphism; for the element $A=-I$ used in this paper, this action is $(x,y,z)\mapsto(-x,-y,z)$.
The Lie algebra of $G$ is the semidirect product $\mathfrak{g}=\mathfrak{nil}\rtimes\so(2)$ of the Heisenberg algebra $\mathfrak{nil}$ and the algebra $\so(2)$ of rotations.
Regarding $X,Y,Z$ above as a basis of $\mathfrak{nil}$ and adjoining a generator $R$ of $\so(2)$, we obtain a basis $\{X, Y, Z, R\}$ of $\mathfrak{g}$ with non-trivial brackets
\[
    [X,Y] = Z, \quad [R, X] = Y, \quad [R,Y] = -X.
\]

Since $G$ is connected, a symmetric bilinear form $\p=\langle\cdot,\cdot\rangle$ on $\g$ is $\Ad(G)$-invariant if and only if
\begin{equation}
    \langle [u,v],w\rangle + \langle v,[u,w]\rangle = 0 \quad (u,v,w\in\g).
    \label{eq:ad_invariance}
\end{equation}
Solving \eqref{eq:ad_invariance} for the basis $\{X,Y,Z,R\}$, we find that $\p$ is $G$-invariant if and only if all its components vanish except
\begin{equation}
    \langle X,X\rangle=\langle Y,Y\rangle=\langle Z,R\rangle=c_1, \quad \langle R,R\rangle=c_2
    \label{eq:inv_poly_nil}
\end{equation}
for some $c_1,c_2\in\R$.
This computation, and those in Sections~\ref{sec:Sol} and \ref{sec:Euclid}, can also be carried out with the program \cite{InvPolynomial}.
We denote this polynomial by $\p=\p_{c_1,c_2}$.

We begin by recording the flatness conditions.
We write a connection $\omega$ on $P_\rho$ in components as
$\omega = \omega_X X + \omega_Y Y + \omega_Z Z + \omega_R R \in \mathcal{A}^1(P_\rho;\g)$.
A connection $\omega$ is flat if and only if $\diff \omega + (1/2)[\omega,\omega] = 0$.
Using the Lie brackets, this condition is equivalent to
\begin{equation}
    \label{eq:flat_conditions_nil}
    \begin{aligned}
        \diff \omega_X &= -\omega_Y \wedge \omega_R, &\quad
        \diff \omega_Y &= -\omega_R \wedge \omega_X, \\
        \diff \omega_Z &= -\omega_X \wedge \omega_Y, &\quad
        \diff \omega_R &= 0.
    \end{aligned}
\end{equation}

We first observe that the pair consisting of a flat connection and a trivial connection yields no non-trivial invariant.

\begin{prop}
    \label{prop:flat_and_trivial_connection_cs}
    Suppose that $P_\rho$ is trivial, and let $\omega_\triv^s$ be the trivial connection determined by a section $s \colon M \to P_\rho$.
    For any flat connection $\omega_\flat$, write $a \coloneqq \omega_\flat - \omega_\triv^s = a_X X + a_Y Y + a_Z Z + a_R R$ with respect to the trivialization given by $s$.
    Then
    \begin{equation}
        \CS_\p(\omega_\flat,\omega_\triv^s) = c_1\,\diff(a_Z \wedge a_R) \in \mathcal{A}^3(M;\R).
        \label{eq:cs_flat_triv_nil}
    \end{equation}
\end{prop}

\begin{proof}
    Both connections are flat, so \eqref{eq:cs_formula_flat_flat} applies; moreover, the Maurer--Cartan equation $\diff_{\omega_\triv^s}a+(1/2)[a,a]=0$ used to derive that formula also holds.
    Under the trivialization given by $s$ we have $\diff_{\omega_\triv^s}=\diff$, so this gives, for the components of $a$, the same system of equations as \eqref{eq:flat_conditions_nil}.
    A direct computation first gives
    \[
        [a,a] = 2a_Y \wedge a_R\,X + 2a_R \wedge a_X\,Y + 2 a_X \wedge a_Y\,Z ,
    \]
    which has no $R$-component.
    Hence, among the components listed in \eqref{eq:inv_poly_nil}, the component $\langle R,R \rangle = c_2$ does not contribute; nor does $a_Z$, since $\langle Z,R \rangle$ is the only possibly non-zero component of the form $\langle Z,\cdot \rangle$.
    The remaining pairings with $a_X X$, $a_Y Y$, and $a_R R$ are all equal to $2c_1\,a_X \wedge a_Y \wedge a_R$ because $\langle X,X \rangle = \langle Y,Y \rangle = \langle Z,R \rangle = c_1$, and therefore
    \[
        \CS_\p(\omega_\flat,\omega_\triv^s) = -c_1\,a_X \wedge a_Y \wedge a_R .
    \]
    On the other hand, the system above gives $\diff a_Z = -a_X \wedge a_Y$ and $\diff a_R = 0$, so the Leibniz rule yields
    \[
        \diff(a_Z \wedge a_R) = \diff a_Z \wedge a_R = -a_X \wedge a_Y \wedge a_R ,
    \]
    and \eqref{eq:cs_flat_triv_nil} follows.
\end{proof}

Since $M$ is closed, Stokes' theorem shows that the integral of $\CS_\p(\omega_\flat,\omega_\triv^s)$ vanishes.
To obtain a non-trivial invariant we therefore need an auxiliary connection, and in Section~\ref{subsec:nil_connection} we construct such a connection $\omega^\sigma$.

Next, we check that the period group \eqref{eq:gauge_period_P} vanishes.
Since $G/K \cong \Nil$ is contractible, we fix a section $\sigma \colon M \to P_\rho/K$ and a reduced sub-$K$-bundle $\pi_\sigma \colon Q_\sigma \to M$, as in Section~\ref{subsec:CSinv}.
Since $\ad_R$ annihilates $Z$ and $R$ and $K$ is connected, $\Ad(K)$ preserves the subspace spanned by $X$ and $Y$, the line $\R Z$, and the line $\R R$, and acts trivially on the $Z$- and $R$-components.
Evaluating the vertical and equivariance conditions componentwise, we see that for any $B=bR\in\kk$, the components of $\omega_\rho$ satisfy $\omega_\rho^R(B^*) = b$ and $\omega_\rho^Z(B^*) = 0$.
As $\kk$ is abelian, the equivariance condition is equivalent to right invariance, and the pullback of the $R$-component to $Q_\sigma$ defines a connection on the principal $K$-bundle $Q_\sigma$.
This connection is flat by \eqref{eq:flat_conditions_nil}, so Chern--Weil theory shows that the real Euler class of $Q_\sigma$ vanishes in $H^2(M;\R)$, and Corollary~\ref{cor:period_is_zero_criteria}\eqref{item:period_zero_euler} gives $\Lambda_\p(P_\rho) = 0$.

\subsection{Construction of the volume-type Chern--Simons invariant}
\label{subsec:nil_connection}

In this subsection, we explicitly construct an auxiliary connection $\omega^\sigma$ on $P_\rho$ using the section $\sigma$ and the reduced sub-$K$-bundle $Q_\sigma$, and show that it is well-defined.

For the pullback to $Q_\sigma$ of the flat connection $\omega_\rho \in \mathcal{A}^1(P_\rho; \g)$, set
\[
    e^1\coloneqq\iota_\sigma^*\omega_\rho^X,\quad e^2\coloneqq\iota_\sigma^*\omega_\rho^Y,\quad e^3\coloneqq\iota_\sigma^*\omega_\rho^Z,\quad \theta\coloneqq\iota_\sigma^*\omega_\rho^R \in \mathcal{A}^1(Q_\sigma;\R) .
\]
That is, we expand $\iota_\sigma^*\omega_\rho$ in components according to the direct sum decomposition $\g = \nil \oplus \kk$ of vector spaces:
\begin{equation}
    \label{eq:flat_connection_components}
    \iota_\sigma^* \omega_\rho = e^1 X + e^2 Y + e^3 Z + \theta R \in \mathcal{A}^1(Q_\sigma;\g) .
\end{equation}
We use this notation throughout this subsection.
As we saw in the previous subsection, $\theta R$ is a connection on the principal $K$-bundle $Q_\sigma$.
Using it, we construct the connection $(\theta-e^3)R$ on the reduced subbundle $Q_\sigma$.

\begin{prop}
    \label{prop:omega_is_well_defined_connection}
    The form $(\theta-e^3)R \in \mathcal{A}^1(Q_\sigma;\kk)$ is a connection on the principal $K$-bundle $Q_\sigma$.
    Consequently, its unique extension $\omega^\sigma$ is a connection on the principal $G$-bundle $P_\rho$.
\end{prop}

\begin{proof}
    As we saw in the previous subsection, $e^3$ satisfies $e^3(B^*)=0$ and is invariant under the right action of $K$, and hence is the pullback of a $1$-form on $M$.
    Since $\kk$ is abelian, the difference of two connections on the principal $K$-bundle is precisely a $\kk$-valued $1$-form pulled back from the base space; conversely, adding such a form to a connection again yields a connection.
    Hence $(\theta-e^3)R = \theta R - e^3 R$ is a connection on the principal $K$-bundle $Q_\sigma$.
    By the theorem on the unique extension of a connection in a reduced subbundle \cite[Chapter II, Proposition 6.1]{KN}, its unique extension $\omega^\sigma$ is a connection on the principal $G$-bundle.
\end{proof}

Next, we show that the \CSname{} form determined by $\omega^\sigma$ and the flat connection $\omega_\rho$ agrees with the volume form of the representation.
For this we first record an expression for that volume form on $Q_\sigma$.
We normalize the $G$-invariant volume form $\vol_G$ on $G/K\cong \Nil$ such that $\vol_G = X^*\wedge Y^*\wedge Z^*$ at the identity.
By the definition of $\vol_\rho$, the $\Nil$-components $e^1, e^2, e^3$ of $\iota_\sigma^*\omega_\rho$ satisfy
\begin{equation}
    \pi_\sigma^*(\sigma^*\vol_\rho) = e^1\wedge e^2\wedge e^3 \in \mathcal{A}^3(Q_\sigma;\R).
    \label{eq:volume_form_components_nil}
\end{equation}

\begin{prop}
    \label{prop:nil_cs_equals_volume}
    For the connection $\omega^\sigma$ associated with any section $\sigma \colon M \to P_\rho/K$, we have
    \begin{equation}
        \CS_\p(\omega^\sigma,\omega_\rho) = -c_2\,\sigma^*\vol_\rho \in \mathcal{A}^3(M;\R).
        \label{eq:cs_vol_form_nil}
    \end{equation}
\end{prop}

\begin{proof}
    We compute the \CSname{} form on $Q_\sigma$.
    Since $\kk$ is abelian and the restriction of \eqref{eq:flat_conditions_nil} to $Q_\sigma$ gives $\diff e^3=-e^1\wedge e^2$ and $\diff\theta=0$, the pullback to $Q_\sigma$ of the curvature $\Omega^\sigma$ of $\omega^\sigma$ is given by
    \begin{equation}
        \label{eq:curvature_omega_nil}
        \iota_\sigma^* \Omega^\sigma = (e^1 \wedge e^2) R \in \mathcal{A}^2(Q_\sigma;\kk).
    \end{equation}
    Set $a \coloneqq \iota_\sigma^*(\omega^\sigma - \omega_\rho)$.
    The definition of $\omega^\sigma$ and the flatness of $\omega_\rho$ give
    \[
        a = -e^1X -e^2Y -e^3Z -e^3R, \quad
        \iota_\sigma^*(\diff_{\omega_\rho}(\omega^\sigma-\omega_\rho)) = \iota_\sigma^*\Omega^\sigma - \frac{1}{2} [a,a] \, .
    \]
    Substituting these into \eqref{eq:cs_formula_general_l2}, we obtain
    \[
        \iota_\sigma^*\pi^*\CS_\p(\omega^\sigma,\omega_\rho)
        = \langle a,\iota_\sigma^*\Omega^\sigma\rangle - \frac{1}{6}\langle a,[a,a]\rangle ,
    \]
    where the brackets give $[a,a] = 2e^2\wedge e^3\,X + 2e^3\wedge e^1\,Y + 2e^1\wedge e^2\,Z$.
    By the components listed in \eqref{eq:inv_poly_nil}, only $\langle Z,R\rangle=c_1$ and $\langle R,R\rangle=c_2$ contribute to the first term, paired with the expression \eqref{eq:curvature_omega_nil} for the curvature, and that term equals $-(c_1+c_2)\,e^1\wedge e^2\wedge e^3$.
    In the second term, $[a,a]$ has no $R$-component and $\langle Z,R\rangle$ is the only possibly non-zero component of the form $\langle Z,\cdot\rangle$, so the $Z$-component of $a$ does not contribute; the remaining three pairings each give $-2c_1\,e^1\wedge e^2\wedge e^3$, whence
    \[
        \langle a,\iota_\sigma^*\Omega^\sigma\rangle = -(c_1+c_2)\,e^1\wedge e^2\wedge e^3, \quad
        \frac{1}{6}\langle a,[a,a]\rangle = -c_1\,e^1\wedge e^2\wedge e^3.
    \]
    Hence
    \[
        \iota_\sigma^*\pi^*\CS_\p(\omega^\sigma,\omega_\rho) = -c_2\,e^1\wedge e^2\wedge e^3,
    \]
    and the assertion follows from \eqref{eq:volume_form_components_nil} together with the injectivity of the pullback $\pi_\sigma^*$ along a surjective submersion.
\end{proof}

We now state the main theorem of this section.

\begin{thm}
    \label{thm:cs_inv_is_vol_nil}
    Let $\rho \colon \pi_1(M) \to G$ be a representation, and normalize the invariant polynomial with $c_2 = -1$.
    Then the \CSname{} invariant determined by $\omega^\sigma$ and $\omega_\rho$ coincides with the volume of the representation:
    \begin{equation}
    \label{eq:cs_inv_vol_nil}
        \CS(M;\omega^\sigma,\omega_\rho) = \Vol_G(M,\rho) \in \R .
    \end{equation}
    This value is independent of the choice of the section $\sigma$ and invariant under conjugation of $\rho$.
\end{thm}

\begin{proof}
    With the normalization $c_2=-1$, Proposition~\ref{prop:nil_cs_equals_volume} gives $\CS_\p(\omega^\sigma,\omega_\rho)=\sigma^*\vol_\rho$.
    Integrating this over $M$ and using the definition \eqref{eq:representation_volume} of the volume of a representation, we obtain $\int_M\CS_\p(\omega^\sigma,\omega_\rho)=\Vol_G(M,\rho)$.
    Since $\Lambda_\p(P_\rho)=0$ as shown above, the \CSname{} invariant of Definition~\ref{dfn:cs_inv_definition} takes values in $\R$, which proves \eqref{eq:cs_inv_vol_nil}.
    Finally, the volume of a representation is independent of the choice of the section and invariant under conjugation of $\rho$, and hence so is the \CSname{} invariant that agrees with it.
\end{proof}

By Theorem~\ref{thm:cs_inv_is_vol_nil}, this invariant is independent of the choice of the section and depends only on $(M,\rho)$.
From now on, we write $\CSvol{\Nil}(M,\rho) \coloneqq \CS(M;\omega^\sigma,\omega_\rho) \in \R$ and call it the volume-type \CSname{} invariant.
Moreover, when $\rho_\gm$ is the holonomy representation, \eqref{eq:rep_vol_geometric_holonomy} and Theorem~\ref{thm:cs_inv_is_vol_nil} yield the following corollary.

\begin{cor}
    \label{cor:cs_inv_nil_geohol}
    If $M$ supports $\Nil$-geometry with metric $g$ and holonomy representation $\rho_\gm \colon \pi_1(M) \to G$, then the volume form of the standard $\Nil$ metric agrees with $\vol_G = X^*\wedge Y^*\wedge Z^*$ normalized above, and
    \begin{equation}
        \CSvol{\Nil}(M,\rho_\gm)= \Vol_G(M,\rho_\gm) = \Vol(M,g).
    \end{equation}
\end{cor}

\subsection{Examples: closed \texorpdfstring{$3$}{3}-manifolds with \texorpdfstring{$\Nil$}{Nil}-geometry}
\label{subsec:nil_examples}

In this subsection, we apply Theorem~\ref{thm:cs_inv_is_vol_nil} and Corollary~\ref{cor:cs_inv_nil_geohol} to closed manifolds supporting $\Nil$-geometry, and then compute the \CSname{} invariant of the Levi-Civita connection.
Throughout, the invariant polynomial is normalized so that $c_2=-1$.

We consider a discrete subgroup $\Gamma \subset \Isom_e\Nil$ such that $M \coloneqq \Gamma \backslash \Nil$ is a closed manifold, with holonomy representation given by the inclusion $\rho_\gm \colon \pi_1(M) \cong \Gamma \hookrightarrow G$.
All the computations below reduce to the case of a lattice with no rotation part.

We use the coordinates of Section~\ref{subsec:nil_prep}, and write an element of $\Nil=\R^2\times\R$ as a pair $(v,w)$ with $v\in\R^2$ and $w\in\R$.

\begin{lem}
    \label{lem:nil_lattice_volume}
    Let $L\subset \R^2$ be a lattice whose fundamental domain has area $1$, with a positively oriented basis $\ell_1,\ell_2$, let $n\in\Z_{>0}$, and, for $\zeta=(0,1/n)$, set
    \[
        \Gamma_0=\langle (\ell_1,0),(\ell_2,0),\zeta\rangle\subset \Nil .
    \]
    Then $\Gamma_0$ is a lattice in $\Nil$, and $N_n\coloneqq\Gamma_0\backslash\Nil$ is the circle bundle over $T^2=\R^2/L$ with Euler number $n$.
    Moreover, its holonomy representation $\rho_\gm^0\colon \pi_1(N_n)\cong\Gamma_0\hookrightarrow G$ satisfies
    \begin{equation}
        \label{eq:cs_nil_lattice}
        \CSvol{\Nil}(N_n,\rho_\gm^0)=\Vol_G(N_n,\rho_\gm^0)=\frac{1}{n}.
    \end{equation}
\end{lem}

\begin{proof}
    Denote the elements $(\ell_i,0)$ of $\Nil$ by the same symbols $\ell_i$.
    By \eqref{eq:nil_matrix} the commutator is given by $[(\ell,0),(\ell',0)]=(0,\det(\ell,\ell'))$, and since the fundamental domain has area $1$ we have $[\ell_1,\ell_2]=(0,1)=\zeta^n$.
    Hence $\Gamma_0=\langle\ell_1,\ell_2,\zeta\rangle$ is an extension of $L$ by the center $\langle\zeta\rangle$, and therefore a discrete cocompact subgroup of $\Nil$, that is, a lattice; in particular $N_n$ is a closed manifold.
    The quotient by the center $\Nil\to\R^2$ maps $\Gamma_0$ onto $L$, so $N_n$ is a circle bundle over $T^2$ with fiber $\langle\zeta\rangle\backslash\R\cong S^1$, and
    \[
        \pi_1(N_n)\cong\Gamma_0=\langle\, \ell_1,\ell_2,\zeta \mid [\ell_1,\ell_2]=\zeta^n,\ [\ell_1,\zeta]=[\ell_2,\zeta]=1 \,\rangle
    \]
    is the standard presentation of the fundamental group of the oriented circle bundle over $T^2$ with Euler number $n$ \cite[\S4]{Scott}.
    The volume of $N_n$ is $1\cdot(1/n)=1/n$, so Corollary~\ref{cor:cs_inv_nil_geohol} gives \eqref{eq:cs_nil_lattice}.
\end{proof}

A holonomy representation with non-trivial rotation part is obtained by adjoining to this lattice an isometry whose rotation part has finite order.
We treat the simplest case, in which the rotation part has order $2$.

\begin{exm}
    \label{exm:nil_rotation}
    Let $n$ be a positive even integer, take $L=\Z^2$ with basis $\ell_1=(1,0)$ and $\ell_2=(0,1)$ in Lemma~\ref{lem:nil_lattice_volume}, and set
    $\tau\coloneqq\left(\left(0,1/(2n)\right),-I\right)\in\Nil\rtimes K$ and $\Gamma_{2,n}\coloneqq\langle\Gamma_0,\tau\rangle$.
    Since the translation part of $\tau$ is central, conjugation by $\tau$ acts on $\Nil$ as the automorphism $(x,y,z)\mapsto(-x,-y,z)$ induced by $-I$, which sends $\ell_1$ and $\ell_2$ to their inverses and fixes $\zeta$; hence $\Gamma_0$ is preserved.
    As $\tau^2=\zeta\in\Gamma_0$ and $\tau\notin\Gamma_0$, the subgroup $\Gamma_0$ is normal of index $2$.
    By the classification of $3$-dimensional almost-Bieberbach groups \cite[\S5, Item 2]{DIKL}, the group $\Gamma_{2,n}$ is torsion-free if and only if $n$ is even.
    Hence $M_{2,n}\coloneqq\Gamma_{2,n}\backslash\Nil$ is a closed manifold, and $\rho_{\gm}\colon\pi_1(M_{2,n})\cong\Gamma_{2,n}\hookrightarrow\Isom_e\Nil$ is a holonomy representation for which the $K$-component of $\rho_{\gm}(\tau)$ is $-I$.
    Since $N_n\to M_{2,n}$ is an orientation-preserving double cover, the naturality \eqref{eq:rep_vol_naturality} of the volume of a representation, Lemma~\ref{lem:nil_lattice_volume}, and Theorem~\ref{thm:cs_inv_is_vol_nil} give
    \[
        \CSvol{\Nil}(M_{2,n},\rho_{\gm})=\Vol_G(M_{2,n},\rho_{\gm})=\frac{1}{2n}.
    \]
\end{exm}

Next, we compute, for the same manifolds, the invariant $\CS(M,g)$ of Section~\ref{subsec:metric_cs}, and compare it with the value given by the main theorem of this section.

\begin{prop}
    \label{prop:metric_cs_nil}
    Let $\Gamma\subset\Nil$ be a lattice, and equip $M=\Gamma\backslash\Nil$ with the standard $\Nil$ metric $g$.
    Then
    \[
        \CS(M,g) \equiv -\frac{1}{16\pi^2}\Vol(M,g) \pmod{\Z}.
    \]
\end{prop}

\begin{proof}
    The left-invariant frame $(X,Y,Z)$ is orthonormal and positively oriented for $\vol_G$, and all brackets other than $[X,Y]=Z$ vanish.
    The structure constants in Lemma~\ref{lem:metric_cs_left_invariant} are therefore
    \[
        (\lambda_1,\lambda_2,\lambda_3)=(0,0,1), \quad (\mu_1,\mu_2,\mu_3)=(1/2,1/2,-1/2).
    \]
    It suffices to substitute these into \eqref{eq:metric_cs_left_invariant}.
\end{proof}

Combined with Corollary~\ref{cor:cs_inv_nil_geohol}, this gives
\[
    \CS(M,g) \equiv -\frac{1}{16\pi^2}\CSvol{\Nil}(M,\rho_\gm) \pmod{\Z}
\]
for the holonomy representation $\rho_\gm$ of a lattice.

\begin{rmk}
    \label{rmk:metric_cs_nil_rotation}
    For the manifolds $M_{2,n}$ of Example~\ref{exm:nil_rotation}, the holonomy representation has non-trivial rotation part and the left-invariant frame does not descend to the quotient, so Lemma~\ref{lem:metric_cs_left_invariant} does not apply.
    On the other hand, the orientation-preserving double cover $q \colon N_n \to M_{2,n}$ is a local isometry, so pulling back the integrand of \eqref{eq:metric_cs_explicit} gives $\CS(N_n,g)=2\,\CS(M_{2,n},g)$.
    The left-hand side is determined by Proposition~\ref{prop:metric_cs_nil}, but multiplication by $2$ on $\R/\Z$ is not injective, and this relation therefore determines $\CS(M_{2,n},g)$ only up to the indeterminacy $(1/2)\Z/\Z$.
\end{rmk}

\section{Chern--Simons invariants in \texorpdfstring{$\Sol$}{Sol}-geometry}
\label{sec:Sol}

In this section, we show that the volume-type \CSname{} invariant in $\Sol$-geometry equals the volume of the representation (Theorem~\ref{thm:cs_inv_vol_sol}).
We then examine its relation to the \CSname{} invariant of the Levi-Civita connection.
Since $K$ is trivial here, we have $P_\rho/K = P_\rho$, and the section of $P_\rho/K$ from which the auxiliary connection is constructed is a section $s \colon M \to P_\rho$, that is, a trivialization of $P_\rho$.

\subsection{Preliminaries on \texorpdfstring{$\Sol$}{Sol}-geometry}
\label{subsec:sol_prep}

The group $\Sol$ is the semidirect product
\begin{equation}
    \Sol = \R^2\rtimes_\phi\R, \quad
    \phi_z = \begin{pmatrix} e^{-z} & 0 \\ 0 & e^{z} \end{pmatrix} .
    \label{eq:sol_semidirect}
\end{equation}
Let $(x,y,z)$ be the coordinates on $\R^2\times\R$.
Assigning to $(x,y,z)$ the matrix
\begin{equation}
    \begin{pmatrix} e^{-z} & 0 & x \\ 0 & e^{z} & y \\ 0 & 0 & 1 \end{pmatrix} ,
    \label{eq:sol_matrix}
\end{equation}
the multiplication is matrix multiplication, as in \cite[\S4]{Scott}.
We write $X,Y,Z$ for the left-invariant frame in these coordinates.
Since $\det\phi_z=1$, the group $\Sol$ is unimodular, and its left-invariant volume form is
\begin{equation}
    \vol_G = \diff x\wedge\diff y\wedge\diff z \in \mathcal{A}^3(\Sol;\R).
    \label{eq:sol_volume_form}
\end{equation}

We call the left-invariant metric making $X,Y,Z$ orthonormal the \emph{standard metric} on $\Sol$; in coordinates,
\begin{equation}
    g_{\Sol} = e^{2z}\diff x^2+e^{-2z}\diff y^2+\diff z^2 .
    \label{eq:sol_metric}
\end{equation}

Throughout this section, $G=\Isom_e\Sol \cong \Sol$ and $K=\{1\}$ \cite[Theorem 3.3]{HaLee}.
Taking $X, Y, Z$ above as a basis of the Lie algebra $\g = \sol$ of $G$, the non-trivial brackets are
\[
    [X,Z] = X, \quad [Y,Z] = -Y .
\]
Solving \eqref{eq:ad_invariance} for the basis $\{X,Y,Z\}$, we find that a symmetric bilinear form $\p = \langle\cdot,\cdot\rangle$ on $\g$ is $G$-invariant if and only if all its components vanish except
\begin{equation}
    \langle Z,Z\rangle = c
    \label{eq:inv_poly_sol}
\end{equation}
for some $c\in\R$.
We denote this polynomial by $\p=\p_c$.

Since $\Sol \cong \R^3$ is contractible and in particular simply connected, the bundle $P_\rho$ is trivial by \cite[Lemma 2.1]{Freed1}.
A connection $\omega = \omega_X X + \omega_Y Y + \omega_Z Z \in \mathcal{A}^1(P_\rho;\g)$ is flat if and only if $\diff \omega + (1/2)[\omega,\omega] = 0$.
Using the Lie brackets, this condition is equivalent to
\begin{equation}
    \diff\omega_X = -\omega_X \wedge \omega_Z, \quad
    \diff\omega_Y = \omega_Y \wedge \omega_Z, \quad
    \diff\omega_Z = 0.
    \label{eq:flat_conditions_sol}
\end{equation}

As in $\Nil$-geometry, a flat connection and a trivial connection yield no non-trivial invariant.
In $\Sol$-geometry the \CSname{} form itself vanishes.

\begin{prop}
    \label{prop:cs_triv_flat_sol_exact}
    For any section $s \colon M \to P_\rho$ and any flat connection $\omega_\flat$ on $P_\rho$, we have $\CS_\p(\omega_\flat,\omega_\triv^s) = 0$.
\end{prop}

\begin{proof}
    Setting $a \coloneqq \omega_\flat-\omega_\triv^s$, we have $\CS_\p(\omega_\flat,\omega_\triv^s) = -(1/6)\langle a, [a,a] \rangle$ by \eqref{eq:cs_formula_flat_flat}.
    The brackets $[X,Z]=X$, $[Y,Z]=-Y$, and $[X,Y]=0$ show that $[a,a]$ has no $Z$-component.
    Since $\langle Z,Z\rangle$ is the only possibly non-zero component of $\p$, we obtain $\langle a, [a,a] \rangle = 0$.
\end{proof}

We therefore need an auxiliary connection to obtain a non-trivial invariant.
In Section~\ref{subsec:sol_connection}, we construct such a connection $\omega^s$.

Finally, we verify that the period group \eqref{eq:gauge_period_P} vanishes.
Since $G\cong\Sol$ is contractible, we have $H^3(G;\R)=0$ and thus $[\WZ_\p]=0$.
Since $P_\rho$ is also trivial, Corollary~\ref{cor:period_is_zero_criteria}\eqref{item:period_zero_trivial} implies $\Lambda_\p(P_\rho)=0$.

\subsection{Construction of the volume-type Chern--Simons invariant}
\label{subsec:sol_connection}

In this subsection, we construct an auxiliary connection from the flat connection associated with a representation $\rho \colon \pi_1(M) \to G=\Sol$.
Using this connection, we obtain a \CSname{} invariant that equals the volume of the representation.

Since $K= \{1\}$, we have $P_\rho/K = P_\rho$.
Fix a section $s \colon M \to P_\rho$ and write
\[
    s^*\omega_\rho = \alpha X + \beta Y + \gamma Z \in \mathcal{A}^1(M;\g) ,
\]
where $\alpha, \beta, \gamma \in \mathcal{A}^1(M;\R)$.
We use this notation throughout this subsection.
Pulling back \eqref{eq:flat_conditions_sol} by $s$ gives $\diff \alpha = -\alpha \wedge \gamma$, $\diff \beta = \beta \wedge \gamma$, and $\diff \gamma = 0$.
Under the trivialization given by $s$, $\g$-valued $1$-forms on $M$ correspond bijectively to connections on $P_\rho$.
Hence there exists a unique connection $\omega^s$ on $P_\rho$ such that $s^*\omega^s = (\alpha+\beta) Z$.

Next, we show that the \CSname{} form determined by $\omega^s$ and the flat connection $\omega_\rho$ agrees with the volume form of the representation.
We normalize the left-invariant volume form $\vol_G$ of $\Sol$ such that $\vol_G = X^*\wedge Y^*\wedge Z^*$ at the identity.
By the definition of $\vol_\rho$, the components $\alpha, \beta, \gamma$ of $s^*\omega_\rho$ satisfy
\begin{equation}
    s^*\vol_\rho = \alpha\wedge\beta\wedge\gamma \in \mathcal{A}^3(M;\R).
    \label{eq:volume_form_components_sol}
\end{equation}

\begin{prop}
\label{prop:cs_auxiliary_connection_sol}
    For the connection $\omega^s$ associated with any section $s \colon M \to P_\rho$, we have
    \begin{equation}
        \CS_\p(\omega^s,\omega_\rho) = 2c \, s^*\vol_\rho \in \mathcal{A}^3(M;\R).
        \label{eq:cs_vol_form_sol}
    \end{equation}
\end{prop}

\begin{proof}
    Under the trivialization given by $s$, the difference $a \coloneqq \omega^s-\omega_\rho$ is expressed as $-\alpha X-\beta Y+u Z$, where $u \coloneqq \alpha+\beta-\gamma$.
    Since $\omega_\rho$ is flat, we compute $\CS_\p(\omega^s,\omega_\rho)$ by \eqref{eq:cs_formula_flat_l2}.
    Since $[\sol,\sol] \subset \operatorname{span}\{X, Y\}$, neither $[a,a]$ nor $[\omega_\rho,a]$ has a $Z$-component, while $\langle Z,Z\rangle = c$ is the only possibly non-zero component of $\p$.
    Thus $\langle a,[a,a]\rangle = 0$, and in $\langle a,\diff_{\omega_\rho}a\rangle$ only the $Z$-component $\diff u$ of $\diff_{\omega_\rho}a$ contributes, giving $\CS_\p(\omega^s,\omega_\rho) = c\,u\wedge\diff u$.
    From the flatness conditions, we have $\diff u = (\beta-\alpha)\wedge\gamma$.
    Since $\gamma \wedge \gamma = 0$, we find
    \[
        u\wedge\diff u
        = (\alpha+\beta)\wedge(\beta-\alpha)\wedge\gamma
        = 2\alpha\wedge\beta\wedge\gamma
        = 2s^*\vol_\rho .
    \]
    Therefore, \eqref{eq:cs_vol_form_sol} holds.
\end{proof}

We now state the main theorem of this section.

\begin{thm}
\label{thm:cs_inv_vol_sol}
    Let $\rho \colon \pi_1(M) \to G$ be a representation, and normalize the invariant polynomial with $c = 1/2$.
    Then the \CSname{} invariant determined by $\omega^s$ and $\omega_\rho$ coincides with the volume of the representation:
    \begin{equation}
    \label{eq:cs_inv_vol_sol}
        \CS(M;\omega^s,\omega_\rho) = \Vol_G(M,\rho) \in \R \, .
    \end{equation}
    This value is independent of the choice of the section $s$ and invariant under conjugation of $\rho$.
\end{thm}

\begin{proof}
    With the normalization $c=1/2$, Proposition~\ref{prop:cs_auxiliary_connection_sol} implies $\CS_\p(\omega^s,\omega_\rho)=s^*\vol_\rho$.
    Integrating this over $M$ and using the definition \eqref{eq:representation_volume} of the volume of a representation, we obtain $\int_M\CS_\p(\omega^s,\omega_\rho)=\Vol_G(M,\rho)$.
    Since $\Lambda_\p(P_\rho)=0$ as shown above, the \CSname{} invariant of Definition~\ref{dfn:cs_inv_definition} takes values in $\R$, which proves \eqref{eq:cs_inv_vol_sol}.
    Finally, the volume of a representation is independent of the choice of the section and invariant under conjugation of $\rho$, and hence so is the \CSname{} invariant that agrees with it.
\end{proof}

As in the case of $\Nil$-geometry, we write $\CSvol{\Sol}(M,\rho) \coloneqq \CS(M;\omega^s,\omega_\rho) \in \R$ from now on.
Moreover, when $\rho_\gm$ is the holonomy representation, \eqref{eq:rep_vol_geometric_holonomy} and Theorem~\ref{thm:cs_inv_vol_sol} yield the following corollary.

\begin{cor}
    \label{cor:cs_inv_sol_geohol}
    If $M$ supports $\Sol$-geometry with metric $g$ and holonomy representation $\rho_\gm \colon \pi_1(M) \to G$, then the volume form of the standard $\Sol$ metric agrees with $\vol_G = X^*\wedge Y^*\wedge Z^*$ normalized above, and
    \begin{equation}
        \CSvol{\Sol}(M,\rho_\gm)= \Vol_G(M,\rho_\gm) = \Vol(M,g).
    \end{equation}
\end{cor}

\subsection{Examples: closed \texorpdfstring{$3$}{3}-manifolds with \texorpdfstring{$\Sol$}{Sol}-geometry}
\label{subsec:sol_examples}

In this subsection, we apply Theorem~\ref{thm:cs_inv_vol_sol} and Corollary~\ref{cor:cs_inv_sol_geohol} to closed manifolds supporting $\Sol$-geometry.
We then compute the \CSname{} invariants of their Levi-Civita connections.
Throughout this subsection, we normalize the invariant polynomial with $c=1/2$.

A standard construction of a closed manifold supporting $\Sol$-geometry is the mapping torus whose gluing map is an Anosov map of $T^2$ \cite[Theorem 5.5]{Scott}.
We use the coordinates and the semidirect product decomposition \eqref{eq:sol_semidirect} of Section~\ref{subsec:sol_prep}.

Let $A\in \SL_2\Z$ be a matrix with $\trace A>2$, with eigenvalues $\lambda>1>\lambda^{-1}>0$ and corresponding eigenvectors $v_+, v_-\in\R^2$, respectively.
Since $\phi_z$ in \eqref{eq:sol_semidirect} contracts the first coordinate and expands the second, we place the eigenvector for $\lambda^{-1}$ in the first column.
Normalizing the eigenvectors so that $\det(v_-\ v_+)=1$, we set $B\coloneqq(v_-\ v_+)\in \SL_2\R$.
Then
\begin{equation}
    B^{-1}AB = \begin{pmatrix} \lambda^{-1} & 0 \\ 0 & \lambda \end{pmatrix} = \phi_{\ln\lambda}.
    \label{eq:sol_conjugation}
\end{equation}
Since $A$ preserves $\Z^2$, the lattice $L\coloneqq B^{-1}\Z^2\subset\R^2$ is preserved by $\phi_{\ln\lambda}$ via \eqref{eq:sol_conjugation}.
Thus
\[
    \Gamma_A \coloneqq L\rtimes(\ln\lambda)\Z \ \subset\ \Sol\cong\Isom_e\Sol
\]
is a lattice in $\Sol$.
For a lattice $L'\subset\R^n$, we denote the volume of the quotient $\R^n/L'$ by $\covol(L')$.
The quotient $M_A\coloneqq\Gamma_A\backslash\Sol$ is the mapping torus of the automorphism of $T^2$ induced by $A$, which is a closed manifold supporting $\Sol$-geometry \cite[\S4]{Scott}.
Let
\[
    \rho_\gm^A\colon \pi_1(M_A)=\Gamma_A \hookrightarrow \Isom_e\Sol
\]
be its holonomy representation.

\begin{prop}
    \label{prop:sol_geometric_holonomy_mapping_torus}
    In the situation above,
    \begin{equation}
        \CSvol{\Sol}(M_A,\rho_\gm^A) = \Vol_G(M_A,\rho_\gm^A) = \ln\lambda.
        \label{eq:cs_vol_sol_mapping_torus}
    \end{equation}
\end{prop}

\begin{proof}
    A fundamental domain for $\Gamma_A=L\rtimes(\ln\lambda)\Z$ is given by $F_L \times [0,\ln\lambda)$, where $F_L\subset\R^2$ is a fundamental domain for $L$ with area $\covol(L)$.
    Using \eqref{eq:sol_volume_form}, we have
    \[
        \Vol(M_A,g) = \int_0^{\ln\lambda}\covol(L)\,\diff z = \ln\lambda\cdot\covol(L) .
    \]
    Since $L=B^{-1}\Z^2$ with $\det B=1$ and $\covol(\Z^2)=1$, we obtain $\covol(L)=1$.
    Therefore, by Corollary~\ref{cor:cs_inv_sol_geohol}, we have \eqref{eq:cs_vol_sol_mapping_torus}.
\end{proof}

Note that $\covol(L)=1$ depends on the normalization $\det B=1$.
Since the automorphism $(x,y,z)\mapsto(ax,by,z)$ of $\Sol$ changes the covolume of $L$, the value above depends not only on the manifold $M_A$ but also on the choice of $\rho_\gm^A$.

Next, we compute the \CSname{} invariant of the Levi-Civita connection for the same manifolds.

\begin{prop}
    \label{prop:metric_cs_sol}
    Let $\Gamma\subset\Sol$ be a lattice, and equip $M=\Gamma\backslash\Sol$ with the standard $\Sol$ metric $g$.
    Then
    \[
        \CS(M,g) \equiv 0 \pmod{\Z} .
    \]
\end{prop}

\begin{proof}
    The left-invariant frame $(X,Y,Z)$ is orthonormal and positively oriented.
    Replacing it with
    \[
        \frac{X+Y}{\sqrt2}, \quad \frac{-X+Y}{\sqrt2}, \quad Z,
    \]
    we obtain an orthonormal frame as in Lemma~\ref{lem:metric_cs_left_invariant}.
    The structure constants are then
    \[
        (\lambda_1,\lambda_2,\lambda_3) = (-1,1,0), \quad (\mu_1,\mu_2,\mu_3) = (1,-1,0).
    \]
    Since $\lambda_3=\mu_3=0$, the right-hand side of \eqref{eq:metric_cs_left_invariant} vanishes.
\end{proof}

Since $\Isom_e\Sol\cong\Sol$, whenever the image of a holonomy representation is contained in $\Isom_e\Sol$, this image is a lattice in $\Sol$ and the \CSname{} invariant of the Levi-Civita connection vanishes.
This applies to mapping tori such as $M_A$, whose gluing maps are induced by matrices $A\in\SL_2\Z$ with $\trace A>2$.
The remaining closed manifolds supporting $\Sol$-geometry involve other connected components of $\Isom\Sol$ \cite[\S4, Theorem 5.3(i)]{Scott}.

\section{Chern--Simons invariants in Euclidean geometry}
\label{sec:Euclid}

In this section, we study Euclidean geometry.
Here, the period group $\Lambda_\p(G)$ of $G$ is $\Z$ or $0$ for the two normalizations of the invariant polynomial used below, which leads to two \CSname{} invariants.
In Section~\ref{subsec:euclid_flat}, we treat the rotation-type invariant determined by a flat connection and a trivial connection (Theorem~\ref{thm:cs_inv_flat_euclid}).
In Section~\ref{subsec:euclid_connection}, we construct the volume-type invariant that equals the volume of the representation (Theorem~\ref{thm:euclidean_cs_equals_volume}).

\subsection{Preliminaries on Euclidean geometry}
\label{subsec:euclid_prep}

The model space $\E^3=\R^3$ is an additive group with coordinates $(x_1,x_2,x_3)$.
We write $P_1,P_2,P_3$ for the left-invariant frame in these coordinates, which is orthonormal for the standard metric
\begin{equation}
    g_{\E^3}=\diff x_1^2+\diff x_2^2+\diff x_3^2
    \label{eq:euclid_metric}
\end{equation}
as in \cite[\S4]{Scott}.

Throughout this section, $G=\Isom_e\E^3=\SE(3)=\R^{3}\rtimes \SO(3)$ and $K=\SO(3)$ \cite[Theorem 3.1]{HaLee}.
Here $G$ acts on $\E^3$ by $(b,A)\cdot x = Ax+b$, and assigning to $(b,A)$ the matrix
\begin{equation}
    \begin{pmatrix} A & b \\ 0 & 1 \end{pmatrix} ,
    \label{eq:euclid_matrix}
\end{equation}
the multiplication is matrix multiplication.
The Lie algebra of $G$ is the semidirect product $\mathfrak{g}=\mathfrak{se}(3)=\R^{3}\rtimes\mathfrak{so}(3)$.
Regarding $P_{1},P_{2},P_{3}$ above as a basis of $\R^3$ and adjoining generators $\{J_1,J_2,J_3\}$ of $\so(3)$, we obtain a basis of $\g$ whose non-trivial brackets are
\[
    [J_i, J_j] = \varepsilon_{ijk}J_k, \quad
    [J_i, P_j] = \varepsilon_{ijk}P_k, \quad
    [P_i, P_j] = 0 ,
\]
where $\varepsilon$ is the Levi-Civita symbol normalized by $\varepsilon_{123}=1$ and summation over $k$ is understood.

Solving \eqref{eq:ad_invariance} for the basis $\{P_i,J_i\}$, we find that a symmetric bilinear form $\p=\langle\cdot,\cdot\rangle$ on $\g$ is $G$-invariant if and only if all its components vanish except
\begin{equation}
    \langle J_i, J_j \rangle = c_1 \delta_{ij}, \quad \langle P_i, J_j \rangle = c_2 \delta_{ij}
    \label{eq:inv_poly_euclid}
\end{equation}
for some $c_1,c_2\in\R$.
We denote this polynomial by $\p=\p_{c_1,c_2}$.

We first formulate the flatness conditions, which will be used in both Section~\ref{subsec:euclid_flat} and Section~\ref{subsec:euclid_connection}.
We write a connection $\omega$ on $P_\rho$ in components as
$\omega = \sum_{i=1}^{3}\omega^{J_i}J_i + \sum_{i=1}^{3}\omega^{P_i}P_i \in \mathcal{A}^1(P_\rho;\g)$.
A connection $\omega$ is flat if and only if $\diff\omega+(1/2)[\omega,\omega]=0$.
Using the Lie brackets, this condition is equivalent to the following equations for every cyclic permutation $(i,j,k)$ of $(1,2,3)$:
\begin{subequations}
\label{eq:flat_conditions_euclid}
    \begin{align}
        \diff\omega^{J_i} + \omega^{J_j}\wedge\omega^{J_k} &= 0, \label{eq:flat_euclid_J}\\
        \diff\omega^{P_i} + \omega^{J_j}\wedge\omega^{P_k} - \omega^{J_k}\wedge\omega^{P_j} &= 0. \label{eq:flat_euclid_P}
    \end{align}
\end{subequations}

We record the computation of $-(1/6)\langle A,[A,A]\rangle$ for a flat $\g$-valued $1$-form $A$, which appears both in \eqref{eq:cs_formula_flat_flat} and in $\WZ_\p$.

\begin{lem}
    \label{lem:euclid_flat_potential}
    Let
    $A = \theta^1J_1+\theta^2J_2+\theta^3J_3 + e^1P_1+e^2P_2+e^3P_3 \in \mathcal{A}^1(N;\g)$
    be a $\g=\se(3)$-valued $1$-form on a manifold $N$ satisfying the Maurer--Cartan equation $\diff A + (1/2)[A,A] = 0$.
    Then, for the invariant polynomial $\p=\p_{c_1,c_2}$ of \eqref{eq:inv_poly_euclid},
    \[
        -\frac{1}{6}\langle A,[A,A]\rangle
        = -c_1\,\theta^1\wedge\theta^2\wedge\theta^3 -c_2\,\diff(
        \theta^1\wedge e^1
        +\theta^2\wedge e^2
        +\theta^3\wedge e^3
        )
        \in \mathcal{A}^3(N;\R).
    \]
\end{lem}

\begin{proof}
    Since $\R^3$ is abelian and $\langle P_i,P_j\rangle=0$, the only non-zero contributions to $\langle A,[A,A]\rangle$ come from terms with three $\so(3)$-components and from mixed terms with two $\so(3)$-components and one $\R^3$-component.
    Let
    \[
        \Theta \coloneqq \theta^1\wedge\theta^2\wedge e^3 + \theta^2\wedge\theta^3\wedge e^1 + \theta^3\wedge\theta^1\wedge e^2 .
    \]
    Using the Lie brackets and \eqref{eq:inv_poly_euclid}, the first contribution is $6c_1\,\theta^1\wedge\theta^2\wedge\theta^3$, and the second is $6c_2\,\Theta$.
    On the other hand, substituting \eqref{eq:flat_conditions_euclid} into $\diff(\theta^i\wedge e^i)=\diff\theta^i\wedge e^i-\theta^i\wedge\diff e^i$ and summing over $i=1,2,3$, we obtain
    \[
        \diff(\theta^1\wedge e^1 + \theta^2\wedge e^2 + \theta^3\wedge e^3) = \Theta .
    \]
    This proves the assertion.
\end{proof}

Next, we determine the period group.
By \eqref{eq:universal_period}, this reduces to integrating $\WZ_\p$ over the classes in $H_3(G;\Z)$.

\begin{lem}
    \label{lem:period_euclidean}
    The period group of $G=\SE(3)$ is
    \begin{equation}
        \Lambda_\p(G)=8\pi^2c_1\Z ,
        \label{eq:period_euclidean_c1}
    \end{equation}
    and $\WZ_\p$ is exact on $G$ whenever $c_1=0$.
\end{lem}

\begin{proof}
    Since $\SE(3)=\R^3\rtimes\SO(3)\simeq\SO(3)$, we have $H_3(G;\Z)\cong H_3(\SO(3);\Z)\cong\Z$, generated by the fundamental class $[\SO(3)]$.
    The Maurer--Cartan form on $G$ is $\omega_\MC=\theta^1J_1+\theta^2J_2+\theta^3J_3+e^1P_1+e^2P_2+e^3P_3$.
    By Lemma~\ref{lem:euclid_flat_potential},
    \[
        \WZ_\p = -c_1\,\theta^1\wedge\theta^2\wedge\theta^3 - c_2\,\diff(\theta^1\wedge e^1 + \theta^2\wedge e^2 + \theta^3\wedge e^3) \in \mathcal{A}^3(G;\R).
    \]
    Since the generator $[\SO(3)]$ is represented by the subgroup $\SO(3)\subset\SE(3)$, we have $\WZ_\p|_{\SO(3)}=-c_1\,\theta^1\wedge\theta^2\wedge\theta^3$.
    Since $\Vol(\SO(3))=8\pi^2$ with respect to the metric making the $\theta^i$ orthonormal, we obtain $\int_{[\SO(3)]}\WZ_\p=-8\pi^2c_1$, which implies \eqref{eq:period_euclidean_c1}.
\end{proof}

Finally, we fix a reduced subbundle.
Since $G/K \cong \E^3$ is contractible, we fix a section $\sigma \colon M \to P_\rho/K$ and a reduced sub-$K$-bundle $\pi_\sigma \colon Q_\sigma \to M$, as in Section~\ref{subsec:CSinv}.

\subsection{The Chern--Simons invariant determined by flat connections}
\label{subsec:euclid_flat}

In this subsection, we assume that $P_\rho$ is trivial, and compute the \CSname{} form determined by the flat connection $\omega_\rho$ and a trivial connection.
Let $s \colon M \to P_\rho$ be a section, and let $\omega_\triv^s$ be the associated trivial connection.
We write
\[
    A_\rho \coloneqq s^*\omega_\rho
    = \theta^1J_1+\theta^2J_2+\theta^3J_3 + e^1P_1+e^2P_2+e^3P_3 \in \mathcal{A}^1(M;\g) ,
\]
where $\theta^i \coloneqq s^*\omega_\rho^{J_i}$ and $e^i \coloneqq s^*\omega_\rho^{P_i}$ are $1$-forms in $\mathcal{A}^1(M;\R)$.
We use this notation throughout this subsection.

\begin{prop}
    \label{prop:cs_form_euclidean}
    For the invariant polynomial $\p=\p_{c_1,c_2}$,
    \begin{equation}
        \CS_\p(\omega_\rho,\omega_\triv^s)
        = -c_1\,\theta^1\wedge\theta^2\wedge\theta^3 -c_2\,\diff(
        \theta^1\wedge e^1
        +\theta^2\wedge e^2
        +\theta^3\wedge e^3
        ) \in \mathcal{A}^3(M;\R).
    \end{equation}
\end{prop}

\begin{proof}
    The assertion follows by applying Lemma~\ref{lem:euclid_flat_potential} with $A=A_\rho$.
\end{proof}

With this preparation, we obtain the rotation-type \CSname{} invariant.

\begin{thm}
    \label{thm:cs_inv_flat_euclid}
    Let $\rho \colon \pi_1(M) \to G$ be a representation such that $P_\rho$ is trivial, and normalize the invariant polynomial with $c_1 = 1/(8\pi^2)$ and $c_2=0$.
    Then the \CSname{} invariant determined by the ordered pair $(\omega_\rho,\omega_\triv^s)$ is given by
    \begin{equation}
        \CS(M;\omega_\rho,\omega_\triv^s) \equiv -\frac{1}{8\pi^2} \int_M \theta^1 \wedge \theta^2 \wedge \theta^3 \pmod{\Z} . \label{eq:cs_inv_flat_euclid}
    \end{equation}
    This value is independent of the choice of the section $s$ and invariant under conjugation of $\rho$.
\end{thm}

\begin{proof}
    By Proposition~\ref{prop:cs_form_euclidean} with $c_1=1/(8\pi^2)$ and $c_2=0$, integrating $\CS_\p(\omega_\rho,\omega_\triv^s)$ over $M$ gives the right-hand side of \eqref{eq:cs_inv_flat_euclid}.
    From Lemma~\ref{lem:period_euclidean}, we have $\Lambda_\p(G)=8\pi^2c_1\Z=\Z$.
    Since $P_\rho$ is trivial, Lemma~\ref{lem:expanded_period} implies $\Lambda_\p(P_\rho)\subset\Z$.
    Therefore, the \CSname{} invariant is defined modulo $\Z$, and satisfies \eqref{eq:cs_inv_flat_euclid}.
    Two sections of the trivial bundle $P_\rho$ differ by a gauge transformation, so \eqref{eq:cs_inv_gauge_indep} shows that the value does not depend on $s$.
    Finally, for $\rho'=g\rho g^{-1}$ with $g\in G$, the map $h\mapsto g^{-1}h$ on the fiber induces an isomorphism $P_{\rho'}\to P_\rho$ of principal $G$-bundles over $\id_M$ which pulls $\omega_\rho$ back to $\omega_{\rho'}$ and trivial connections back to trivial connections; the invariance under conjugation therefore follows from naturality \eqref{eq:cs_naturality}.
\end{proof}

By Theorem~\ref{thm:cs_inv_flat_euclid}, this invariant is independent of the choice of the section and depends only on $(M,\rho)$.
From now on, we write $\CSrot{\E^3}(M,\rho) \coloneqq \CS(M;\omega_\rho,\omega_\triv^s) \in \R/\Z$, and call it the rotation-type \CSname{} invariant.
For a Euclidean manifold, the holonomy representation satisfies the assumption above, and $\CSrot{\E^3}$ coincides with the invariant $\CS(M,g)$ of Section~\ref{subsec:metric_cs}.

\begin{cor}
    \label{cor:cs_inv_flat_euclid_geometryhol}
    If $M$ supports Euclidean geometry with metric $g$ and holonomy representation $\rho_\gm \colon \pi_1(M) \to G$, then $P_{\rho_\gm}$ is trivial and, for the \CSname{} invariant \eqref{eq:metric_cs_explicit} of the Levi-Civita connection,
    \[
        \CSrot{\E^3}(M,\rho_\gm) = \CS(M,g) \in \R/\Z.
    \]
\end{cor}

\begin{proof}
    The section $\sigma_D \colon M \to P_{\rho_\gm}/K$ determined by the developing map gives a reduced sub-$K$-bundle $Q_{\sigma_D}$, on which the $\R^3$-components $e=(e^1,e^2,e^3)$ of $\iota_{\sigma_D}^*\omega_{\rho_\gm}$ form the solder form for the metric $g$ induced by $D$: at each $q\in Q_{\sigma_D}$, the map $e_q\colon T_{\pi(q)}M\to\R^3$ is an orientation-preserving isometry.
    Hence $q\mapsto e_q^{-1}$ identifies $Q_{\sigma_D}$ with the orthonormal frame bundle $\FrameBundle$ of $(M,g)$.
    Since $\FrameBundle$ is trivial as shown in Section~\ref{subsec:metric_cs}, the bundle $P_{\rho_\gm} \cong Q_{\sigma_D}\times_K G$ is also trivial, and Theorem~\ref{thm:cs_inv_flat_euclid} applies.
    We choose $s$ to be a section of $Q_{\sigma_D}$, which corresponds to an orthonormal frame of $(M,g)$.
    The translation equation \eqref{eq:flat_euclid_P} then corresponds to the torsion-free condition in the first structure equation.
    Thus the connection with components $\theta^i$ is the Levi-Civita connection $\omega_\LC$.
    Furthermore, the restriction of $\p$ with $c_1=1/(8\pi^2)$ to $\so(3)$ agrees with $\p=(1/2)P_1$ in Remark~\ref{rmk:metric_cs_as_relative}.
    Therefore, $\CS_\p(\omega_{\rho_\gm},\omega_\triv^s)$ equals the \CSname{} form determined by $\omega_\LC$ and $\omega_\triv^s$ on $Q_{\sigma_D}=\FrameBundle$, and the invariant in Theorem~\ref{thm:cs_inv_flat_euclid} coincides with the \CSname{} invariant \eqref{eq:metric_cs_explicit} of the Levi-Civita connection.
\end{proof}

\subsection{Construction of the volume-type Chern--Simons invariant}
\label{subsec:euclid_connection}

In this subsection, we construct a \CSname{} invariant that equals the volume of the representation, without assuming that $P_\rho$ is trivial.
We use the section $\sigma$ and the reduced sub-$K$-bundle $Q_\sigma$ fixed in Section~\ref{subsec:euclid_prep}, and set $\theta^i \coloneqq \iota_\sigma^*\omega_\rho^{J_i}$ and $e^i \coloneqq \iota_\sigma^*\omega_\rho^{P_i}$ in $\mathcal{A}^1(Q_\sigma;\R)$.
We use this notation throughout this subsection.

We now construct a connection whose \CSname{} form equals the volume form of the representation.
Set $\widehat{e} \coloneqq \sum_{i=1}^3 e^i J_i \in \mathcal{A}^1(Q_\sigma;\kk)$.
By Proposition~\ref{prop:omega_is_well_defined_connection_euclid} below, this form defines an $\ad P_\rho$-valued $1$-form on $M$; we denote its pullback to $P_\rho$ by the same symbol $\widehat{e}$, and define a $1$-form $\omega^\sigma$ on $P_\rho$ by
\[
    \omega^\sigma \coloneqq \omega_\rho+\widehat{e} \in \mathcal{A}^1(P_\rho;\g) .
\]

\begin{prop}
    \label{prop:omega_is_well_defined_connection_euclid}
    The form $\widehat{e}$ defines an $\ad P_\rho$-valued $1$-form on $M$.
    Consequently, $\omega^\sigma$ is a connection on the principal $G$-bundle $P_\rho$.
\end{prop}

\begin{proof}
    For any $B \in \kk$, the vertical condition for $\omega_\rho$ gives $\omega_\rho(B^*)=B \in \kk$.
    Hence $e^i(B^*)=0$, and $\widehat{e}$ is horizontal.
    Since $\kk$ is a subalgebra and $\R^3$ is an ideal, the decomposition $\g=\kk\oplus\R^3$ is $\Ad(K)$-invariant.
    By the equivariance of $\omega_\rho$, the triple $(e^1,e^2,e^3)$ transforms under $\Ad(K)|_{\R^3}$, which is the standard representation of $\SO(3)$.
    The structure constants of $[J_i,P_j]$ and $[J_i,J_j]$ are both $\varepsilon_{ijk}$, so the linear map $P_i \mapsto J_i$ is $\kk$-equivariant.
    Since $K$ is connected, this map is an $\Ad(K)$-equivariant isomorphism $\R^3 \to \kk$.
    Thus $\widehat{e}$ is horizontal and $\Ad(K)$-equivariant, and descends to a $1$-form on $M$ with values in $\ad P_\rho = Q_\sigma\times_K\g$.
    Since the sum of a connection and an $\ad P_\rho$-valued $1$-form is a connection, $\omega^\sigma$ is a connection on $P_\rho$.
\end{proof}

Next, we show that the \CSname{} form determined by $\omega^\sigma$ and $\omega_\rho$ agrees with the volume form of the representation.
We normalize the $G$-invariant volume form $\vol_G$ on $G/K \cong \E^3$ such that $\vol_G = P_1^*\wedge P_2^*\wedge P_3^*$ at the identity.
By the definition of $\vol_\rho$, the $\R^3$-components $e^i$ of $\iota_\sigma^*\omega_\rho$ satisfy
\begin{equation}
    \pi_\sigma^*(\sigma^*\vol_\rho) = e^1\wedge e^2\wedge e^3 \in \mathcal{A}^3(Q_\sigma;\R).
    \label{eq:volume_form_components_euclid}
\end{equation}

\begin{prop}
    \label{prop:cs_auxiliary_connection_euclidean}
    For the connection $\omega^\sigma$ associated with any section $\sigma \colon M \to P_\rho/K$ and the invariant polynomial $\p=\p_{c_1,c_2}$, we have
    \begin{equation}
        \CS_\p(\omega^\sigma,\omega_\rho) = (2c_1+6c_2)\,\sigma^*\vol_\rho \in \mathcal{A}^3(M;\R).
        \label{eq:cs_vol_form_euclid}
    \end{equation}
\end{prop}

\begin{proof}
    Set $a\coloneqq\omega^\sigma-\omega_\rho$ and compute \eqref{eq:cs_formula_flat_l2} on $Q_\sigma$.
    Writing $A\coloneqq\iota_\sigma^*\omega_\rho$, we have $\iota_\sigma^*\diff_{\omega_\rho}a=\diff\widehat{e}+[A,\widehat{e}]$, whose $\so(3)$-component vanishes by the translation part \eqref{eq:flat_euclid_P} of the flatness condition.
    By the Lie brackets, the remaining $\R^3$-component and $[\widehat{e},\widehat{e}]$ are
    \begin{align*}
        \iota_\sigma^*\diff_{\omega_\rho}a &= 2e^2\wedge e^3 P_1 + 2e^3\wedge e^1 P_2 + 2e^1\wedge e^2 P_3, \\
        [\widehat{e},\widehat{e}] &= 2e^2\wedge e^3 J_1 + 2e^3\wedge e^1 J_2 + 2e^1\wedge e^2 J_3 .
    \end{align*}
    Hence the components $\langle P_i,J_j\rangle=c_2\delta_{ij}$ and $\langle J_i,J_j\rangle=c_1\delta_{ij}$ of \eqref{eq:inv_poly_euclid} give
    \[
        \langle \widehat{e},\iota_\sigma^*\diff_{\omega_\rho}a\rangle = 6c_2\,e^1\wedge e^2\wedge e^3,
        \quad
        \frac{1}{3}\langle \widehat{e},[\widehat{e},\widehat{e}]\rangle = 2c_1\,e^1\wedge e^2\wedge e^3 ,
    \]
    which implies
    \[
        \pi_\sigma^*\CS_\p(\omega^\sigma,\omega_\rho) = (2c_1+6c_2)\,e^1\wedge e^2\wedge e^3 .
    \]
    The assertion follows from the injectivity of the pullback $\pi_\sigma^*$ along a surjective submersion together with \eqref{eq:volume_form_components_euclid}.
\end{proof}

We now state the main theorem of this section.

\begin{thm}
    \label{thm:euclidean_cs_equals_volume}
    Let $\rho \colon \pi_1(M) \to G$ be a representation, and normalize the invariant polynomial with $c_1=0$ and $c_2=1/6$.
    Then the \CSname{} invariant determined by $\omega^\sigma$ and $\omega_\rho$ coincides with the volume of the representation:
    \begin{equation}
        \CS(M;\omega^\sigma,\omega_\rho) = \Vol_G(M,\rho) \in \R . \label{eq:cs_inv_volume_euclid}
    \end{equation}
    This value is independent of the choice of the section $\sigma$ and invariant under conjugation of $\rho$.
\end{thm}

\begin{proof}
    By Proposition~\ref{prop:cs_auxiliary_connection_euclidean} with $c_1=0$ and $c_2=1/6$, we have
    $\CS_\p(\omega^\sigma,\omega_\rho)=\sigma^*\vol_\rho$.
    Integrating this over $M$ yields the right-hand side of \eqref{eq:cs_inv_volume_euclid}.
    Since $c_1=0$, Lemma~\ref{lem:period_euclidean} implies $[\WZ_\p]=0$ in $H^3(G;\R)$.
    Moreover, $K=\SO(3)\cong\R P^3$ satisfies $H^1(K;\R)=H^2(K;\R)=0$.
    Hence Corollary~\ref{cor:period_is_zero_criteria}\eqref{item:period_zero_cohomology} gives $\Lambda_\p(P_\rho)=0$.
    Thus the \CSname{} invariant of Definition~\ref{dfn:cs_inv_definition} takes values in $\R$ and equals $\Vol_G(M,\rho)$.
    Finally, the volume of a representation is independent of the choice of the section and invariant under conjugation of $\rho$, and hence so is the \CSname{} invariant that agrees with it.
\end{proof}

As in $\Nil$- and $\Sol$-geometry, we write $\CSvol{\E^3}(M,\rho) \coloneqq \CS(M;\omega^\sigma,\omega_\rho) \in \R$ from now on.
Moreover, when $\rho_\gm$ is the holonomy representation, \eqref{eq:rep_vol_geometric_holonomy} and Theorem~\ref{thm:euclidean_cs_equals_volume} yield the following corollary.

\begin{cor}
    \label{cor:euclidean_cs_vol_geohom}
    If $M$ supports Euclidean geometry with metric $g$ and holonomy representation $\rho_\gm \colon \pi_1(M) \to G$, then the volume form of the standard Euclidean metric agrees with $\vol_G = P_1^*\wedge P_2^*\wedge P_3^*$ normalized above, and
    \begin{equation}
        \CSvol{\E^3}(M,\rho_\gm)= \Vol_G(M,\rho_\gm) = \Vol(M,g).
    \end{equation}
\end{cor}

\subsection{Examples: flat manifolds}
\label{subsec:euclid_examples}

In this subsection, we apply Corollary~\ref{cor:euclidean_cs_vol_geohom} to closed $3$-manifolds supporting Euclidean geometry (flat manifolds), and compute the \CSname{} invariants of their Levi-Civita connections.

We first review some standard facts on flat manifolds.
Let $(M,g)$ be a closed oriented flat $3$-manifold with holonomy representation $\rho_\gm$, and set
\[
    \Gamma \coloneqq \rho_\gm(\pi_1(M)) \subset \SE(3), \quad
    L \coloneqq \Gamma\cap\R^3, \quad
    H \coloneqq \Gamma/L, \quad
    h \coloneqq |H| .
\]
We call $L$ the translation subgroup and $H$ the linear holonomy group; writing $\mathrm{rot}\colon\SE(3)\to\SO(3)$ for the projection to the rotation part, $L$ is the kernel of $\mathrm{rot}|_\Gamma$ and $H\cong \mathrm{rot}(\Gamma)$, so that $H$ is the image of the rotation part of $\rho_\gm$ and is to be distinguished from $\Gamma$ itself.
By the first Bieberbach theorem, $L$ is a lattice of rank $3$ in $\R^3$ and has finite index in $\Gamma$ \cite[Theorem 3.2.1]{Wolf}.
In particular, $H$ is a finite group of order $h = [\Gamma:L]$.
By the Bieberbach classification, there are exactly six closed orientable flat $3$-manifolds up to diffeomorphism, distinguished by their linear holonomy groups $H$.
Table~\ref{tab:flat_manifolds_holonomy} lists the group $H$, its order $h$, the first homology group $H_1(M;\Z)$ \cite[Theorem 3.5.5, Corollary 3.5.10]{Wolf}, and the eta invariant $\eta(M,g)$ \cite[Example 1]{Szczepanski}.

\begin{table}[ht]
    \caption{The linear holonomy group $H$, its order $h$, the first homology group, and the eta invariant of the closed orientable flat $3$-manifolds.}
    \label{tab:flat_manifolds_holonomy}
    \setlength{\tabcolsep}{4pt}
    \begin{tabular}{c|cccccc}
        $H$         & $1$    & $\Z/2$             & $\Z/3$         & $\Z/4$         & $\Z/6$ & $(\Z/2)^2$  \\ \hline
        $h$         & $1$    & $2$                & $3$            & $4$            & $6$    & $4$         \\
        $H_1(M;\Z)$ & $\Z^3$ & $\Z\oplus(\Z/2)^2$ & $\Z\oplus\Z/3$ & $\Z\oplus\Z/2$ & $\Z$   & $(\Z/4)^2$  \\
        $\eta(M,g)$ & $0$    & $0$                & $-2/3$         & $-1$           & $-4/3$ & $0$
    \end{tabular}
\end{table}

With these facts, we first compute the volume-type \CSname{} invariant.

\begin{prop}
    \label{prop:euclid_flat_manifolds}
    In the notation above,
    \[
        \CSvol{\E^3}(M,\rho_\gm) = \Vol(M,g) = \frac{\covol(L)}{h}.
    \]
\end{prop}

\begin{proof}
    Set $T_L \coloneqq \R^3/L$.
    Since $L$ is normal in $\Gamma$, the projection $q \colon T_L \to M$ is an $h$-fold covering.
    Because $q$ is a local isometry, we have $\covol(L) = \Vol(T_L,g) = h\,\Vol(M,g)$.
    Applying Corollary~\ref{cor:euclidean_cs_vol_geohom} completes the proof.
\end{proof}

Rescaling all lengths by a factor $t>0$, that is, replacing $g$ by $t^2g$, scales the lattice $L$ by $t$ and its covolume by $t^3$.
For each of the six manifolds, we can normalize the metric such that $\covol(L)=1$, in which case the volume becomes $1/h$.
However, this is a metric normalization; the value $1/h$ is not an invariant of the diffeomorphism type of $M$ alone.
Since the volume-type \CSname{} invariant coincides with the Riemannian volume, it scales accordingly.

Next, we compute the \CSname{} invariant of the Levi-Civita connection for flat manifolds.
Unlike the $\Nil$ and $\Sol$ cases, a flat manifold with $H\neq 1$ is not a quotient of $\E^3$ by a lattice, so Lemma~\ref{lem:metric_cs_left_invariant} does not apply with $X=\E^3$; we treat all six manifolds at once through the eta invariant and \eqref{eq:aps_relation}.

\begin{prop}
    \label{prop:metric_cs_euclid}
    Let $(M,g)$ be a closed oriented flat $3$-manifold whose eta invariant is, up to sign, as listed in Table~\ref{tab:flat_manifolds_holonomy}.
    Then
    \[
        \CS(M,g) = 0 \in \R/\Z .
    \]
\end{prop}

\begin{proof}
    Table~\ref{tab:flat_manifolds_holonomy} gives $H_1(M;\Z)$, from which the values of $\sigma_2(M)$
    for the six types are $0, 2, 0, 1, 0, 2$, respectively.
    Combining these with $\eta(M,g)$ in Table~\ref{tab:flat_manifolds_holonomy}, the right-hand side of \eqref{eq:aps_relation} yields $0, 1, -1, -1, -2, 1$, which are all integers.
    For the opposite sign of $\eta(M,g)$, the values are $0, 1, 1, 2, 2, 1$, which are also integers.
    Thus $\CS(M,g) \equiv 0 \pmod{\Z}$.
\end{proof}

Together with Corollary~\ref{cor:cs_inv_flat_euclid_geometryhol}, this implies that the rotation-type \CSname{} invariant $\CSrot{\E^3}(M,\rho_\gm)$ also vanishes for the flat metrics treated in Proposition~\ref{prop:metric_cs_euclid}.

\medskip
\noindent\textbf{Acknowledgments.}
The author is deeply grateful to his supervisor, Takefumi Nosaka, for many valuable discussions and constant encouragement throughout this work.

\medskip
\noindent\textbf{Use of AI tools.}
During the preparation of this paper, the author used LLMs for English translation, language editing, and assistance with \LaTeX{} typesetting.
All mathematical content was developed by the author, who takes full intellectual responsibility for the contents of this paper.

\end{document}